\documentclass[11pt]{article}

\usepackage{amsfonts}
\usepackage{amsmath,amsthm,amscd,amssymb,mathrsfs,setspace, accents,textcomp,mathtools,caption,footmisc}
\usepackage{latexsym,epsf,epsfig}
\usepackage{cite}
\usepackage{color}
\usepackage[hmargin=1.0in,vmargin=1in]{geometry}
\usepackage[T1]{fontenc}
\usepackage[utf8]{inputenc}
\usepackage{combelow}
\usepackage{newunicodechar}
\usepackage{tikz}
\usepackage{mathtools}
\usepackage{tabularx}
\usepackage{blindtext}
\usepackage[unicode,breaklinks=true,colorlinks=true,linkcolor=black,urlcolor=black,citecolor=black]{hyperref}
\usepackage{soul}

  \chardef\forshowkeys=0
  \chardef\refcheck=0
  \chardef\showllabel=0
  \chardef\sketches=0
  \chardef\showcolors=0

\ifnum\forshowkeys=1
  
  \usepackage[notref,notcite,color]{showkeys}
\fi

\ifnum\refcheck=1
  \usepackage{refcheck}
\fi

\newcommand\QQ{B}

\newcommand{\R}{\mathbb{R}}

\newcommand{\Z}{ \mathbb{Z}}
\newcommand{\N}{ \mathbb{N}}
\newcommand{\eqnb}{ \begin{equation}}
\newcommand{\eqne}{ \end{equation}}

\theoremstyle{plain}
\newtheorem{theorem}{Theorem}[section]
\newtheorem{lemma}[theorem]{Lemma}
\newtheorem{proposition}[theorem]{Proposition}

\newtheorem{corollary}[theorem]{Corollary}

\theoremstyle{remark}
\newtheorem{remark}[theorem]{Remark}
\newtheorem{definition}[theorem]{Definition}

\def\lec{\lesssim}
\def\gec{\gtrsim}

\def\llabel#1{\marginnote{\color{lightgray}\rm\small(#1)}[-0.0cm]\notag}
\def\llabel#1{\notag}

\numberwithin{equation}{section}
\definecolor{mygray}{rgb}{.6,.6,.6}
\definecolor{myblue}{rgb}{9, 0, 1}
\definecolor{colorforkeys}{rgb}{1.0,0.0,0.0}
\newlength\mytemplen
\newsavebox\mytempbox
\makeatletter
\newcommand\mybluebox{%
\@ifnextchar[
{\@mybluebox}%
{\@mybluebox[0pt]}}
\def\@mybluebox[#1]{%
\@ifnextchar[
{\@@mybluebox[#1]}%
{\@@mybluebox[#1][0pt]}}
\def\@@mybluebox[#1][#2]#3{
\sbox\mytempbox{#3}%
\mytemplen\ht\mytempbox
\advance\mytemplen #1\relax
\ht\mytempbox\mytemplen
\mytemplen\dp\mytempbox
\advance\mytemplen #2\relax
\dp\mytempbox\mytemplen
\colorbox{myblue}{\hspace{1em}\usebox{\mytempbox}\hspace{1em}}}
\makeatother
\def\loc{\text{loc}}
\def\uloc{\text{uloc}}

\def\ee{\mathrm{e}}

\def\div{\mathop{\rm div}\nolimits}

\def\dist{\mathop{\rm dist}\nolimits}
\def\card{\mathop{\rm card}\nolimits}  
\def\supp{\mathop{\rm supp}\nolimits}
\def\indeq{\quad{}}

\definecolor{colorigor}{rgb}{1, 0.2, 0.8}
\definecolor{colorabhi}{rgb}{1, 0.5, 0}

\def\colb{\color{black}}

\definecolor{colorgggg}{rgb}{0.1,0.5,0.3}
\definecolor{colorllll}{rgb}{0.0,0.7,0.0}
\definecolor{colorhhhh}{rgb}{0.3,0.75,0.4}
\definecolor{colorpppp}{rgb}{0.7,0.0,0.2}
\definecolor{coloroooo}{rgb}{0.45,0.0,0.0}
\definecolor{colorqqqq}{rgb}{0.1,0.7,0}

\def\cole{\colb}

\definecolor{coloraaaa}{rgb}{0.6,0.6,0.6}

\def\comma{ {\rm ,\qquad{}} }

\def\fractext#1#2{{#1}/{#2}}

\newcommand{\p}{\partial}

\newcommand{\If}{I_{\text{far}}}
\newcommand{\In}{I_{\text{near}}}
\renewcommand{\sc}{\mathscr{C}}
\renewcommand{\d}{\mathrm{d}}

\DeclareMathOperator*{\esssup}{ess\,sup}

\def\comma{ {\rm ,\qquad{}} }        

\begin{document}
\baselineskip=5.8truemm
\title{A unified theory of existence of suitable weak solutions to the 3D~incompressible Navier-Stokes equations for non-decaying initial data}

\author{
  Abhishek Balakrishna\thanks{University of Southern California, Los Angeles, CA, email: ab45315@usc.edu.}%
  \and
  Igor Kukavica\thanks{University of Southern California, Los Angeles, CA, email: kukavica@usc.edu.}
  \and
  Wojciech O\.za\'nski\thanks{Florida State University, Tallahassee, FL, and Princeton University, Princeton, NJ, email: wozanski@fsu.edu.}
}

\maketitle
\begin{abstract}
We consider any cover $\mathscr{C}$ of $\R^3$ by balls of radius bigger or equal $1$ satisfying two conditions: (i)~any ball intersects at most $\sigma>0$ other balls, and (ii)~intersecting balls have comparable sizes. We consider a natural Morrey-type space such that the $L^2_{\uloc}$ setting
of Lemari\'e-Rieusset (\emph{Recent Developments in the Navier–Stokes Problem},  2002) and the dyadic-type space considered by Bradshaw and Kukavica (\emph{J. Math. Fluid Mech.}, 22(1), 2020) are particular cases. We provide a~priori estimates and prove local existence of weak solutions in two cases; first, when there exists $\epsilon>0$ such that $|B|^{1/3} \lec |x_B|^{1-\epsilon}$ for all $B\in \mathscr{C}$, where $x_B$ denotes  the center of~$B$,
or when
$|B|^{1/3} \gec 1+ |x_B|$ for all~$B\in\mathscr{C}$.
In particular,  we introduce a new non-divergence-free approach to the construction of weak solutions, which simplifies the existence proof in the  $L^2_{\mathrm{uloc}}$ setting. In addition, for the dyadic setting,
we do not require vanishing at the spatial infinity. The constructed solutions are suitable in the sense of Caffarelli, Kohn, and Nirenberg, thus allowing an application of the partial regularity theory.
\end{abstract}

\section{Introduction}
We consider the 3D~incompressible  Navier-Stokes equation (NSE) on $\mathbb{R}^3\times(0,T)$, for $T>0$:
\begin{equation}\label{nse}
\begin{split}
	\frac{\partial u}{\partial t}+(u\cdot\nabla)u-\Delta u+\nabla p&=f,\\
   	 \div u&=0,
	 \end{split}
\end{equation}
where we set the viscosity to~$1$.
Here, $u$ denotes the fluid velocity and $p$ denotes the pressure function. The fluid velocity evolves from a divergence-free initial data $u_0\colon \mathbb{R}^3\to\mathbb{R}^3$. In his seminal work, Leray \cite{L} (see~\cite{OP} for a comprehensive modern review) established the existence of global in time weak solutions for~\eqref{nse}
satisfying the strong energy inequality
\begin{equation*}
        \frac{1}{2}
	\int_{\mathbb{R}^3}|u(t)|^2
	+
       \int_s^t\int_{\mathbb{R}^3}|\nabla u|^2
       \leq
       \frac{1}{2}
       \int_{\mathbb{R}^3}|u (s) |^2
\end{equation*}
for almost all $s\geq 0$,
including $s=0$,
and all $t\geq s$.

Since then,  many works aimed at  weakening this notion of a solution in order to obtain existence with more general data. In particular, non-decaying solutions have been considered. We mention the recent studies by Abe and Giga\cite{A,A2,AG,AG2}, concerning the Stokes and Navier-Stokes equations in $L^\infty$, Gallay and Slijepčević \cite{GS}, on the boundedness of the 2D Navier-Stokes equations, and Maremonti and Shimizu \cite{MS}, as well as Kwon and Tsai \cite{KwT}, on global weak solutions with initial data with slowly decaying oscillations at  infinity.
 See also~\cite{BK2, K2, KV} on local existence of weak bounded solutions. Lemari\'e-Rieusset pioneered in \cite{LR} the notion of local energy solutions on~$\mathbb{R}^3$, which are concerned with initial data $u_0\in L^2_{\uloc}(\R^3)$. The work was later extended by Kikuchi and Seregin~\cite{kikuchi}. Motivated by \cite{16,kikuchi,LR,BK1}, we introduce in the next definition the notion of local energy solutions that is of interest.

\begin{definition}\label{locenergy}
	A divergence-free vector field $u\in L^2_{\loc}(\mathbb{R}^3\times[0,T))$ is a \emph{local energy solution}  to~\eqref{nse} with divergence-free initial data $u_0\in L^2_{\loc}(\mathbb{R}^3)$ if:
	\begin{enumerate}
		\item[(i)] $u\in L^\infty(0,T;L^2_{\loc})\cap
		L^2_{\loc}(\mathbb{R}^3\times[0,T])$
with		
		$\nabla u \in L^2_\uloc$,
		\item[(ii)] for some $p\in L^{3/2}_{\loc}(\mathbb{R}^3\times(0,T))$, the pair $(u,p)$ is a distributional solution to \eqref{nse},
\item[(iii)] the function $t\mapsto \int u(x,t)\cdot w(x)\,\d x$ is continuous on $[0, T )$ for any compactly supported $w\in L^2(\mathbb{R}^3)$,  with
$u(\cdot,0)=u_0$,
\item[(iv)] $(u,p)$ is a suitable solution in the sense of Caffarelli–Kohn–Nirenberg, i.e., for all cylinders  $U$ compactly supported in $\mathbb{R}^3\times[0,T)$ and all non-negative $\phi\in C_0^\infty(U)$, we have the local energy inequality
			\begin{align}
  			\begin{split}
&		\int |u (t) |^2 \phi (t) +	2\int_0^t \int |\nabla u|^2\phi
        \\&\indeq
			\leq \int |u_0 |^2 \phi (0) +
			\int_0^t \int |u|^2(\partial_t\phi
			+ \Delta\phi)+\int_0^t \int(|u|^2+2p)(u\cdot\nabla\phi)
  \end{split}
\label{localenergy}
\end{align}
for almost every $t\in (0,T)$,
		\item[(v)] for every ball $\QQ\subset\mathbb{R}^3$, there exist sets $\QQ^*$, $\QQ^{**}$ such that $\QQ \Subset \QQ^* \Subset \QQ^{**}$, and a function $p_\QQ(t)\in  L^{3/2}((0,T))$ such that for $x\in \QQ^{*}$ and $0<t<T$,
		\begin{equation}\label{localpress_intro}
		\begin{split}
		p(x,t)-p_\QQ(t)
			 &=
     			 -|u(x,t)|^2 
      			+\text{p.v.}\int_{y\in \QQ^{**}}K_{ij}(x-y)(u_i(y,t)u_j(y,t))\d y
      			\\&\indeq
      			+\int_{y\notin \QQ^{**}}\left(K_{ij}(x-y)-K_{ij}(x_\QQ-y)\right)(u_i(y,t)u_j(y,t))\d y,
			\end{split}
		\end{equation}
where $x_\QQ$ is the center of $\QQ$ and $K_{ij}(y)=\partial_i\partial_j(4\pi|y|)^{-1}$. 
	\end{enumerate}
\end{definition}
Note that the notion of a solution in Definition~\ref{locenergy} is different from the usual definition of a local energy solution or a local Leray solution, as we assume neither that $\nabla u \in L^2 L^2_{\uloc}$, nor that $u_0\in L^2_{\uloc}$. Local energy solutions with these additional assumptions are henceforth referred to as local Leray solutions. As noted in \cite{MMP,BK1}, in addition to being a more general class of solutions, this definition retains enough structure to facilitate useful developments in several theoretical problems of interest. Here, we briefly mention a few.

Local energy solutions are suitable in the sense of Caffarelli–Kohn–Nirenberg, and one may therefore apply their partial regularity results~\cite{CKN}. Additionally, local Leray solutions appear as limits of rescaled solutions to the 3D NSE near potential singularities. The energy, being supercritical with respect to the scaling $u_\lambda(x,t)=\lambda u(\lambda x,\lambda^2 t)$, blows up. The limit quantity still solves the 3D NSE in the local Leray class. Lastly, local Leray solutions were key to the work of Jia and \v Sver\'ak \cite{JS}  regarding the construction of forward self-similar solutions with large initial data. This seminal work, along with \cite{GSv,JS2} played an important role in the understanding of non-uniqueness of Leray-Hopf solutions.

In this paper, we establish local-in-time existence of suitable weak solutions for initial data in a general set of subspaces of~$L^2_{\loc}$. To be precise we have the following.
\begin{definition}\label{Cr}
Given arbitrary constants $\sigma, \eta>0$, 
denote by $\mathscr{C}=\mathscr{C}^{\sigma,\eta }$ a family of closed balls $\QQ$ such that $\bigcup_{\QQ\in \mathscr{C} } \QQ =\R^3$, $|\QQ|\geq 4\pi /3$ and:
\begin{enumerate}
\item[(i)] each ball in $\mathscr{C}$ intersects at most $\sigma$ balls in $\mathscr{C}$, 
\item[(ii)] if $\QQ, \QQ'\in \mathscr{C}$ intersect, then
\[
\frac1\eta \leq \frac{|\QQ|^{\frac13}}{|\QQ'|^{\frac13}} \leq \eta .\]
\end{enumerate} 
\end{definition}
The above definition allows various covers of $\R^3$ by balls, each of them satisfying the asymptotic bound
\eqnb\label{asympt}
|B|^{1/3} \lec |x_B| \qquad \text{ as }|x_B|\to \infty
;
\eqne
see~Corollary~\ref{cor_forintro} for a proof. 
\begin{definition}[weighted space $M^{p,q}$]
Let $\mathscr{C}$ be a collection of balls as in Definition~\ref{Cr}.  We say that $f\in M^{p,q}=M^{p,q}_{\mathscr{C}}$ if
     \begin{equation}\label{mpq}
       \|f\|_{M^{p,q}_{\mathscr{C}}}^p
       \coloneqq \sup_{\QQ\in \mathscr{C}}\frac{1}{|\QQ|^{q/3}}\int_{\QQ}|f|^p
       <\infty
     .
     \end{equation}
     We also set
     \[
     M\coloneqq M^{2,2}_{\mathscr{C}}.
     \]
\end{definition}
We will use the short-hand notation
\eqnb\label{alphabeta}
\alpha_t [ v ] \coloneqq \displaystyle\esssup_{0<s<t}\| v (s)\|_{M}^2,\qquad \beta_t [v] \coloneqq \sup_{Q\in\sc}\frac{1}{|Q|^{2/3}}\int_0^t\int_Q|\nabla v|^2
\eqne
for any vector field~$v$, and 
\eqnb\label{alphabetas}
\alpha_t \coloneqq \alpha_t [ u ],\qquad \beta_t \coloneqq \beta_t [u] .
\eqne
We prove an a~priori estimate for all solutions to \eqref{nse} for which $\alpha_t+\beta_t$ is continuous.

\begin{theorem}[A~priori estimate in $M$]\label{thm_apriori}
Let $u_0\in M$ be divergence-free, and suppose that $(u,p)$ is a local energy solution with initial data $u_0$, where 
	\[
		T
		= \frac{1}{C}\min
		\left\lbrace 1,\|u_0\|_{M}^{-4}\right\rbrace
	\]
for a sufficiently large constant~$C=C(\sigma,\eta)\geq 1$. Suppose that $\alpha_t$ and $\beta_t$ are continuous on $[0,T]$. Then,
    \eqnb\label{main_apr}
    	\esssup_{0<t<T} \|u(t)\|_{M}^2+\sup_{\QQ\in\sc}\frac{1}{\QQ^{2/3}}\int_0^T\int_{\QQ}|\nabla u|^2
	\lec_{\sigma,\eta } \|u_0\|_{M}.
    \eqne
\end{theorem}
In order to establish local-in-time existence of solutions satisfying the a~priori bound \eqref{main_apr}, we make a technical restriction on the family $\mathscr{C}$:  either the growth at infinity stated in \eqref{asympt} holds \emph{exactly} or we have a \emph{strictly slower} growth. Namely, we assume that either
\eqnb\label{ass2}
|B|^{1/3} \gec 1+ |x_B| \quad \text{ for all }B\in \sc
\eqne
or  there exists $\epsilon \in (0,1)$ such that 
\eqnb\label{ass_extra}
|B|^{1/3} \lec |x_B|^{1-\epsilon }\qquad \text{ for all }B\in \sc .
\eqne

With this, we now state our main result.

\begin{theorem}[Local existence in $M$]\label{thm_main}
Assume $\sc$ satisfies \eqref{ass2} or \eqref{ass_extra}, let $u_0\in M$ be divergence-free, and let $T=C^{-1}\min\{1,\|u_0\|_{M}^{-4}\}$, where $C=C(\sigma , \eta )\geq 1$ is a sufficiently large constant. Then there exist a local energy solution $(u,p)$  to the Navier–Stokes equations \eqref{nse} satisfying~\eqref{main_apr}.
\end{theorem}

We emphasize that we do not assume \emph{any decay} of the initial data as $|x|\to \infty$. Moreover, our Definition~\ref{Cr} does not assume \emph{any structure} of the cover, such as uniform or dyadic tiling. In this sense, Theorem~\ref{thm_main} is a broad generalization of the local existence result in $L^2_{\uloc}$, due to Lemari\'e-Rieuseet~\cite{LR}. In fact, Definition~\ref{Cr} also allows the dyadic tiling, as considered by \cite{BK1} (see also Bradshaw, Kukavica, and O\.za\'nski \cite{BKO} in the case of half-space $\R^3_+$). However, the additional assumption \eqref{ass_extra} excludes the borderline case of~\cite{BK1}. The assumption  \eqref{ass_extra} is due to our new method of construction, which uses pointwise approximation of $u_0$ using \emph{non-divergence-free} cutoffs (which we describe in Section~\ref{sec_sketch} below); the assumption \eqref{ass_extra} appears necessary to control the pressure function in such an approximation. To be precise, the  pressure decomposition \eqref{localpressapprox} of the approximations necessarily results in a borderline-divergent term. Namely, $I_4$ in~\eqref{localpressapprox} gives rise to logarithmically diverging summation in~\eqref{EQx}.

\subsection{Examples of configurations covered by Definition~\ref{Cr}}

As we vary $\mathscr{C}$, we obtain a variety of subspaces of~$L^2_{\loc}$. In this section, we point out a few examples of $\mathscr{C}$ that satisfy the conditions of Definition~\ref{Cr} and describe the resulting subspace.\\

\begin{enumerate}
\item[(i)] As mentioned above, Lemari\'e-Rieusset showed in~\cite{LR} the existence of suitable weak solutions in~$L^2_{\uloc}(\R^3)$. Choosing $\mathscr{C}$ appropriately, $M$ becomes equivalent to~$L^2_{\uloc}(\R^3)$. The appropriate $\mathscr{C}$ can be obtained by first covering $\R^3$ by unit cubes, and then replacing each unit cube with a ball circumscribing the unit cube. Such a collection of balls satisfies Definition~\ref{Cr} with $\eta=1$ and $\sigma=26$. Since all balls are of the same size, the $M$-norm is equivalent to the $L^2_{\uloc}(\R^3)$-norm.
	
\item[(ii)] In \cite{BK1}, the authors show existence of suitable weak solutions with initial data $u_0 \in {M}$ with decay at spatial infinity,
  \begin{equation}\label{decay}
    \frac{1}{|\QQ|^{2/3}}\int_{\QQ}|u_0 |^2 \to 0 \quad \text{~as~}x_\QQ\to\infty,  \QQ\in \mathscr{C},
  \end{equation} 
and in the particular case of  cubes doubling in size as we move away from the origin. This configuration is also covered by Definition~\ref{Cr} with $\eta=2$ and $\sigma=4$. Again, one can replace cubes with balls circumscribing the cube.
	
\item[(iii)]
In \cite{BC}, the authors consider a family of intermediate spaces $F^p_{\alpha,\lambda}$ with $1\leq p<\infty$ and $0\leq\alpha\leq 1$. This space consists of functions $f\in L^p_{\loc}(\R^3)$ with finite norm $\|f\|_{F^p_{\alpha,\lambda}}$, where for any $R>1$,
	\begin{equation}\label{F}
		\|f\|_{F^p_{\alpha,\lambda}}=\sup_{x\in\R^3}\left(\frac{1}{r_x^\lambda}\int_{B(x,r_x)}|f|^p\right)^{1/p}\comma r_x=\frac{1}{8}\max\{R,|x|\}^\alpha.
	\end{equation}
	Note that when $\alpha=0$, we get $L^2_{\uloc}(\R^3)$, and when $\alpha=1$, we obtain $M^{p,\lambda}_{\mathscr{C}}$ for any $\eta>1$. Both spaces are covered by Definition~\ref{Cr}. For $0<\alpha<1$, observe that the supremum is taken over balls whose radii grow at a rate of $\alpha<1$. From Definition~\ref{Cr}(ii), we may conclude that since the case of $\alpha=1$ is covered by Definition~\ref{Cr} for any $\eta>1$, so is the case with slower radial growth rate. 
	
	\item[(iv)] One may also consider covers of $\R^3$ where the balls grow in size only in certain directions. This would mean that the norm is measured over larger scales in those directions. An example of such a cover would be one where the balls double in size as we move along the $z$ axis and away from the origin, as shown in the figure below. Such covers have not been considered in the past and are well-adapted for axi-symmetric solutions. 
	
	\begin{center}

\tikzset{every picture/.style={line width=0.75pt}} 

\begin{tikzpicture}[x=0.75pt,y=0.75pt,yscale=-0.7,xscale=0.7]

\draw    (353,224) -- (620,225.49) ;
\draw [shift={(622,225.5)}, rotate = 180.32] [color={rgb, 255:red, 0; green, 0; blue, 0 }  ][line width=0.75]    (10.93,-3.29) .. controls (6.95,-1.4) and (3.31,-0.3) .. (0,0) .. controls (3.31,0.3) and (6.95,1.4) .. (10.93,3.29)   ;
\draw    (354,225) -- (353.01,5.5) ;
\draw [shift={(353,3.5)}, rotate = 89.74] [color={rgb, 255:red, 0; green, 0; blue, 0 }  ][line width=0.75]    (10.93,-3.29) .. controls (6.95,-1.4) and (3.31,-0.3) .. (0,0) .. controls (3.31,0.3) and (6.95,1.4) .. (10.93,3.29)   ;
\draw    (353,224) -- (81,224.5) ;
\draw [shift={(79,224.5)}, rotate = 359.9] [color={rgb, 255:red, 0; green, 0; blue, 0 }  ][line width=0.75]    (10.93,-3.29) .. controls (6.95,-1.4) and (3.31,-0.3) .. (0,0) .. controls (3.31,0.3) and (6.95,1.4) .. (10.93,3.29)   ;
\draw   (155,225) .. controls (155,211.19) and (166.19,200) .. (180,200) .. controls (193.81,200) and (205,211.19) .. (205,225) .. controls (205,238.81) and (193.81,250) .. (180,250) .. controls (166.19,250) and (155,238.81) .. (155,225) -- cycle ;
\draw   (205,225) .. controls (205,211.19) and (216.19,200) .. (230,200) .. controls (243.81,200) and (255,211.19) .. (255,225) .. controls (255,238.81) and (243.81,250) .. (230,250) .. controls (216.19,250) and (205,238.81) .. (205,225) -- cycle ;
\draw   (254,225) .. controls (254,211.19) and (265.19,200) .. (279,200) .. controls (292.81,200) and (304,211.19) .. (304,225) .. controls (304,238.81) and (292.81,250) .. (279,250) .. controls (265.19,250) and (254,238.81) .. (254,225) -- cycle ;
\draw   (304,225) .. controls (304,211.19) and (315.19,200) .. (329,200) .. controls (342.81,200) and (354,211.19) .. (354,225) .. controls (354,238.81) and (342.81,250) .. (329,250) .. controls (315.19,250) and (304,238.81) .. (304,225) -- cycle ;
\draw   (354,225) .. controls (354,211.19) and (365.19,200) .. (379,200) .. controls (392.81,200) and (404,211.19) .. (404,225) .. controls (404,238.81) and (392.81,250) .. (379,250) .. controls (365.19,250) and (354,238.81) .. (354,225) -- cycle ;
\draw   (404,225) .. controls (404,211.19) and (415.19,200) .. (429,200) .. controls (442.81,200) and (454,211.19) .. (454,225) .. controls (454,238.81) and (442.81,250) .. (429,250) .. controls (415.19,250) and (404,238.81) .. (404,225) -- cycle ;
\draw   (453,225) .. controls (453,211.19) and (464.19,200) .. (478,200) .. controls (491.81,200) and (503,211.19) .. (503,225) .. controls (503,238.81) and (491.81,250) .. (478,250) .. controls (464.19,250) and (453,238.81) .. (453,225) -- cycle ;
\draw   (503,225) .. controls (503,211.19) and (514.19,200) .. (528,200) .. controls (541.81,200) and (553,211.19) .. (553,225) .. controls (553,238.81) and (541.81,250) .. (528,250) .. controls (514.19,250) and (503,238.81) .. (503,225) -- cycle ;
\draw   (167,187) .. controls (167,166.01) and (184.01,149) .. (205,149) .. controls (225.99,149) and (243,166.01) .. (243,187) .. controls (243,207.99) and (225.99,225) .. (205,225) .. controls (184.01,225) and (167,207.99) .. (167,187) -- cycle ;
\draw   (217,187) .. controls (217,166.01) and (234.01,149) .. (255,149) .. controls (275.99,149) and (293,166.01) .. (293,187) .. controls (293,207.99) and (275.99,225) .. (255,225) .. controls (234.01,225) and (217,207.99) .. (217,187) -- cycle ;
\draw   (266,187) .. controls (266,166.01) and (283.01,149) .. (304,149) .. controls (324.99,149) and (342,166.01) .. (342,187) .. controls (342,207.99) and (324.99,225) .. (304,225) .. controls (283.01,225) and (266,207.99) .. (266,187) -- cycle ;
\draw   (315,186) .. controls (315,165.01) and (332.01,148) .. (353,148) .. controls (373.99,148) and (391,165.01) .. (391,186) .. controls (391,206.99) and (373.99,224) .. (353,224) .. controls (332.01,224) and (315,206.99) .. (315,186) -- cycle ;
\draw   (366,187) .. controls (366,166.01) and (383.01,149) .. (404,149) .. controls (424.99,149) and (442,166.01) .. (442,187) .. controls (442,207.99) and (424.99,225) .. (404,225) .. controls (383.01,225) and (366,207.99) .. (366,187) -- cycle ;
\draw   (416,187) .. controls (416,166.01) and (433.01,149) .. (454,149) .. controls (474.99,149) and (492,166.01) .. (492,187) .. controls (492,207.99) and (474.99,225) .. (454,225) .. controls (433.01,225) and (416,207.99) .. (416,187) -- cycle ;
\draw   (465,187) .. controls (465,166.01) and (482.01,149) .. (503,149) .. controls (523.99,149) and (541,166.01) .. (541,187) .. controls (541,207.99) and (523.99,225) .. (503,225) .. controls (482.01,225) and (465,207.99) .. (465,187) -- cycle ;
\draw   (117,187) .. controls (117,166.01) and (134.01,149) .. (155,149) .. controls (175.99,149) and (193,166.01) .. (193,187) .. controls (193,207.99) and (175.99,225) .. (155,225) .. controls (134.01,225) and (117,207.99) .. (117,187) -- cycle ;
\draw   (515,187) .. controls (515,166.01) and (532.01,149) .. (553,149) .. controls (573.99,149) and (591,166.01) .. (591,187) .. controls (591,207.99) and (573.99,225) .. (553,225) .. controls (532.01,225) and (515,207.99) .. (515,187) -- cycle ;
\draw   (101,96) .. controls (101,52.37) and (136.37,17) .. (180,17) .. controls (223.63,17) and (259,52.37) .. (259,96) .. controls (259,139.63) and (223.63,175) .. (180,175) .. controls (136.37,175) and (101,139.63) .. (101,96) -- cycle ;
\draw   (201,97) .. controls (201,53.37) and (236.37,18) .. (280,18) .. controls (323.63,18) and (359,53.37) .. (359,97) .. controls (359,140.63) and (323.63,176) .. (280,176) .. controls (236.37,176) and (201,140.63) .. (201,97) -- cycle ;
\draw   (300,97) .. controls (300,53.37) and (335.37,18) .. (379,18) .. controls (422.63,18) and (458,53.37) .. (458,97) .. controls (458,140.63) and (422.63,176) .. (379,176) .. controls (335.37,176) and (300,140.63) .. (300,97) -- cycle ;
\draw   (400,98) .. controls (400,54.37) and (435.37,19) .. (479,19) .. controls (522.63,19) and (558,54.37) .. (558,98) .. controls (558,141.63) and (522.63,177) .. (479,177) .. controls (435.37,177) and (400,141.63) .. (400,98) -- cycle ;
\draw    (353.25,226.07) -- (355.8,416.55) ;
\draw [shift={(355.83,418.55)}, rotate = 269.23] [color={rgb, 255:red, 0; green, 0; blue, 0 }  ][line width=0.75]    (10.93,-3.29) .. controls (6.95,-1.4) and (3.31,-0.3) .. (0,0) .. controls (3.31,0.3) and (6.95,1.4) .. (10.93,3.29)   ;
\draw   (540.56,262.53) .. controls (540.73,283.52) and (523.85,300.67) .. (502.87,300.84) .. controls (481.88,301.02) and (464.73,284.14) .. (464.56,263.16) .. controls (464.39,242.17) and (481.26,225.02) .. (502.24,224.85) .. controls (523.23,224.67) and (540.38,241.55) .. (540.56,262.53) -- cycle ;
\draw   (490.56,262.94) .. controls (490.73,283.93) and (473.86,301.08) .. (452.87,301.25) .. controls (431.88,301.43) and (414.73,284.55) .. (414.56,263.57) .. controls (414.39,242.58) and (431.26,225.43) .. (452.25,225.26) .. controls (473.23,225.08) and (490.38,241.96) .. (490.56,262.94) -- cycle ;
\draw   (441.56,263.35) .. controls (441.73,284.33) and (424.86,301.48) .. (403.87,301.66) .. controls (382.89,301.83) and (365.73,284.96) .. (365.56,263.97) .. controls (365.39,242.98) and (382.26,225.83) .. (403.25,225.66) .. controls (424.23,225.49) and (441.39,242.36) .. (441.56,263.35) -- cycle ;
\draw   (392.57,264.75) .. controls (392.74,285.73) and (375.87,302.89) .. (354.88,303.06) .. controls (333.9,303.23) and (316.74,286.36) .. (316.57,265.37) .. controls (316.4,244.38) and (333.27,227.23) .. (354.26,227.06) .. controls (375.24,226.89) and (392.4,243.76) .. (392.57,264.75) -- cycle ;
\draw   (341.56,264.17) .. controls (341.73,285.15) and (324.86,302.3) .. (303.87,302.48) .. controls (282.89,302.65) and (265.74,285.78) .. (265.56,264.79) .. controls (265.39,243.8) and (282.27,226.65) .. (303.25,226.48) .. controls (324.24,226.31) and (341.39,243.18) .. (341.56,264.17) -- cycle ;
\draw   (291.56,264.58) .. controls (291.74,285.56) and (274.86,302.71) .. (253.88,302.89) .. controls (232.89,303.06) and (215.74,286.19) .. (215.57,265.2) .. controls (215.39,244.21) and (232.27,227.06) .. (253.25,226.89) .. controls (274.24,226.72) and (291.39,243.59) .. (291.56,264.58) -- cycle ;
\draw   (242.57,264.98) .. controls (242.74,285.96) and (225.86,303.12) .. (204.88,303.29) .. controls (183.89,303.46) and (166.74,286.59) .. (166.57,265.6) .. controls (166.4,244.62) and (183.27,227.46) .. (204.25,227.29) .. controls (225.24,227.12) and (242.39,243.99) .. (242.57,264.98) -- cycle ;
\draw   (590.55,262.12) .. controls (590.73,283.11) and (573.85,300.26) .. (552.87,300.43) .. controls (531.88,300.61) and (514.73,283.73) .. (514.56,262.75) .. controls (514.38,241.76) and (531.26,224.61) .. (552.24,224.44) .. controls (573.23,224.26) and (590.38,241.14) .. (590.55,262.12) -- cycle ;
\draw   (192.57,265.39) .. controls (192.74,286.37) and (175.87,303.53) .. (154.88,303.7) .. controls (133.89,303.87) and (116.74,287) .. (116.57,266.01) .. controls (116.4,245.03) and (133.27,227.87) .. (154.26,227.7) .. controls (175.24,227.53) and (192.39,244.4) .. (192.57,265.39) -- cycle ;
\draw   (607.3,352.99) .. controls (607.66,396.62) and (572.58,432.28) .. (528.95,432.63) .. controls (485.32,432.99) and (449.66,397.91) .. (449.3,354.29) .. controls (448.95,310.66) and (484.02,275) .. (527.65,274.64) .. controls (571.28,274.28) and (606.94,309.36) .. (607.3,352.99) -- cycle ;
\draw   (507.29,352.81) .. controls (507.65,396.44) and (472.57,432.1) .. (428.95,432.45) .. controls (385.32,432.81) and (349.66,397.73) .. (349.3,354.11) .. controls (348.94,310.48) and (384.02,274.82) .. (427.65,274.46) .. controls (471.28,274.1) and (506.94,309.18) .. (507.29,352.81) -- cycle ;
\draw   (408.3,353.62) .. controls (408.66,397.25) and (373.58,432.91) .. (329.95,433.27) .. controls (286.32,433.62) and (250.66,398.55) .. (250.3,354.92) .. controls (249.95,311.29) and (285.02,275.63) .. (328.65,275.27) .. controls (372.28,274.91) and (407.94,309.99) .. (408.3,353.62) -- cycle ;
\draw   (308.29,353.44) .. controls (308.65,397.07) and (273.57,432.73) .. (229.94,433.09) .. controls (186.31,433.45) and (150.66,398.37) .. (150.3,354.74) .. controls (149.94,311.11) and (185.02,275.45) .. (228.65,275.09) .. controls (272.28,274.73) and (307.94,309.81) .. (308.29,353.44) -- cycle ;

\draw (333,4.4) node [anchor=north west][inner sep=0.75pt]    {$z$};
\draw (611,235.4) node [anchor=north west][inner sep=0.75pt]    {$x$};

\end{tikzpicture}

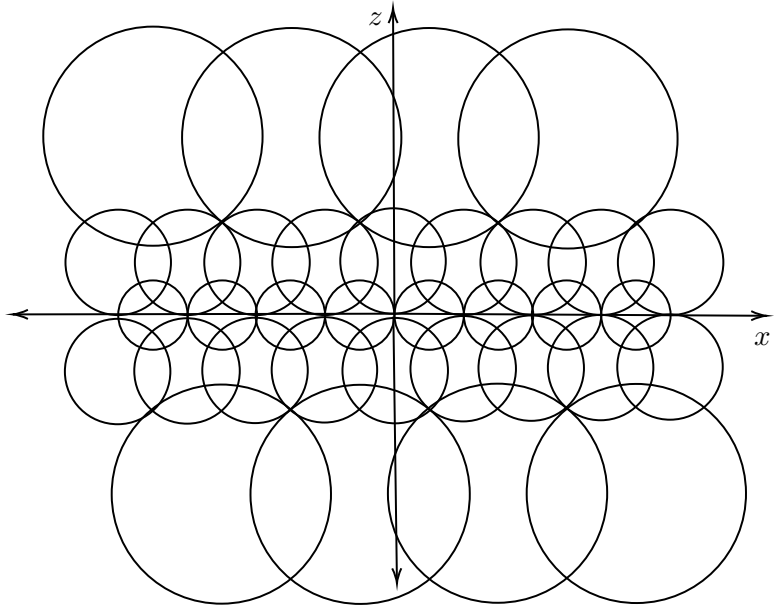
\captionof{figure}{Balls grow only along the $z$ direction}\label{example}
	\end{center} 
\end{enumerate}

\begin{remark}
As illustrated by the examples above, Definition~\ref{Cr} encompasses a broad class of intermediate spaces situated between \( L^p_{\uloc}(\mathbb{R}^3) \) and \( L^p_{\loc}(\mathbb{R}^3) \). A central difficulty in establishing results for this family of spaces lies in the need to handle a wide spectrum of spatial scales. Importantly, the balls in \( \mathscr{C} \) are permitted to expand or contract in arbitrary directions. While \cite{BC, BK1} consider balls that grow isotropically and \cite{LR} focuses on balls of fixed radius, addressing anisotropic growth---where balls can both enlarge and shrink differently along each axis---the proof requires a fundamentally different strategy. To manage the diversity of scales involved, we reformulated key estimates on a sufficiently fine, uniform mesh over \( \mathbb{R}^3 \), which provided a consistent framework to carry out the necessary analysis across the entire family of spaces.

This scale variability also introduces substantial obstacles when attempting to extend local existence results to global ones. The method in \cite{LR} depends critically on uniform control at a fixed scale---something we cannot guarantee when the underlying balls grow unevenly. On the other hand, the approach in \cite{BC} relies on the fact that balls expand rapidly as one moves away from the origin, a property that likewise does not hold in our setting. As a result, both of these existing frameworks become inapplicable in our context.
\end{remark}

\subsection{Sketch of the proof of Theorem~\ref{thm_main}}\label{sec_sketch}

We first consider the case \eqref{ass2}, i.e., $|B|^{1/3} \gec |x_B|$. This lower bound implies that for any sufficiently large $n$ there are only finitely many (independently of $n$) balls $B\in \sc$ intersecting the annulus $B(2n)\setminus B(n)$; see Lemma~\ref{lem_finmany} for a proof. This gives a natural way of approximating any $u_0\in M$ by an $L^2$ divergence-free vector field. Indeed, one may simply apply a cutoff function $Z_n$ such that $\supp \, \nabla Z_n \subset B(2n)\setminus B(n)$, and then apply the Bogovski\u{\i}~\cite{B} theorem to $\div (u_0 Z_n )=u_0 \cdot \nabla Z_n$ to obtain a divergence-free approximation which is globally square integrable on~$\R^3$. This way we can use the classical  Leray theory on existence of weak solutions, and take the limit $n\to \infty$ to obtain the claim of Theorem~\ref{thm_main}.
 There are two main factors that make this possible. First, the $M$-norm of the resulting approximation can be estimated by $\| u_0 \|_M$ because only finitely many $B\in \sc$ intersect $B(2n)\setminus B(n)$.
Second, this way, the approximation satisfies $g=u_0$ on $B(n)$, rather than an approximation in the $M$-norm as $n\to \infty$. This implies that any $u_0\in M$ can be approximated, and not only those $u_0$ in the closure (in the $M$-norm) of~${L^2}$.

This issue becomes much more difficult in the case of~\eqref{ass_extra}. In this case, as $|x_B|\to \infty$, we have no control on the number of balls $B\in \sc$ intersecting any annulus. An easy example to see this is the $\mathrm{uloc}$ setting (i.e., when $B\in \sc \Leftrightarrow B=k+ B(0,1)$, $k\in \Z^3$), in which case an annulus $B(2n)\setminus B(n)$ requires a cover of at least $O(n^3)$ balls from~$\sc$. Even if the annulus is thinner, say $B(n+1)\setminus B(n)$, the cover must have at least $O(n^2)$ balls from $\sc$, which is impossible to handle using a Bogovski\u{\i}-type approach.

In order to deal with this, we devise a completely new procedure of constructing solutions based on a non-divergence-free existence theory. To this end, we consider an auxiliary linear system,
   \begin{equation}\label{v_system_intro}
\begin{split}
	\p_t  \widetilde{u}+(v\cdot\nabla)\widetilde{u}-\Delta \widetilde{u}+\nabla \widetilde{p}&=0,
	\\
   	\widetilde{u}(0)=\widetilde{u}_0,
\end{split}
\end{equation} 
where $v\in L^\infty_t W^{1,\infty}_x$ is a given vector field such that $v=0$ for $|x|\geq n$, and $\widetilde{p}$ satisfies, at each time, the Poisson equation
\eqnb\label{tildep_eq}
\Delta \widetilde{p} = \p_i \p_j (v_j u_i ) - \p_i (u_i \div \, v)  
\eqne
on~$\R^3$. Thanks to this $w\coloneqq \div \, \widetilde{u}$ satisfies the heat equation. In particular, $w$ admits very fast spatial decay, which is, however, only partially helpful.\\

The main challenge is to obtain efficient a~priori bounds in~$M$. Namely, the pressure decomposition~\eqref{tildep_eq} now involves a lower order term involving $\div v$ and a locally integrable kernel (see~\eqref{EQ08} for details). The kernel decays very slowly at far distances, and the main challenge is to handle this lack of decay together with the very limited knowledge of the collection~$\sc$. In fact, the two conditions in Definition~\ref{Cr} give extremely limited information on the structure of the cover of $\R^3$ by balls from $\sc$ for large distances. To this end, we develop several combinatorial tools which enable us to close the pressure estimate on $\widetilde{p}$ with the help of the condition \eqref{ass_extra}; see Section~\ref{sec_lin_p} for details.\\

Another major challenge is to construct solutions using the system~\eqref{tildep_eq}. To this end we divide $[0,T]$ (where $T$ is from Theorem~\ref{thm_apriori}) into $K\in \N$ pieces, and we use \eqref{tildep_eq} with $v^{(0)} \coloneqq 0$ on $[0,T/K]$ to obtain the first iterate~$\widetilde{u}^{(1)}$. We then cut this iterate at radius $n$ and mollify with a mollifier of length-scale $\gamma <  1$ to obtain~$v^{(1)}$. This way we can use $v^{(1)}$ to obtain the second iterate $\widetilde{u}^{(2)}$, on time interval $[T/K,2T/K]$. Continuing such a procedure, we end up with a solution $\widetilde{u}$ to \eqref{tildep_eq} on $[0,T]$ with $v$ of the form of a \emph{retarded mollification} of $\widetilde{u}$ and such that $v=0$ for $|x|\geq n$. The aim is then to take the limit as $n\to \infty$, and then $K\to \infty$ and $\gamma \to \infty$, but the scheme contains two subtle pitfalls into one can fall in this approach. 

The first difficulty is to obtain a uniform a~priori estimate in~$M$. To this end, one cannot simply construct $\widetilde{u}$ one time interval after the previous one, as suggested above, since this would result in an exponential accumulation of a constant in the estimate in the $M$-norm. Instead, at each iteration one needs to solve an initial-value problem from $t=0$. This means that, when finding a $k$-th iteration (i.e., on the time interval $[(k-1)T/K,kT/K]$), one must redefine all $v^{(l)}$'s for $l\leq k$ to find the next iterate $\widetilde{u}^{(k)}$ on $[0,kT/K]$. In other words, this issue is related to an error that accumulates every time the local energy inequality is invoked to obtain an  a~priori estimate in~$M$.\\

The second difficulty is a careful treatment of constants in the $M$-type estimates on the $v^{(k)}$'s and the $\widetilde{u}^{(k)}$'s. Namely, supposing that 
\[\alpha_t [\widetilde{u}^{(1)}]+ \beta_t [\widetilde{u}^{(1)}] \leq C_1
\]
for $t\in [0,T/K]$, the cutoff and mollification  give that 
\[\alpha_t [v^{(1)}]+ \beta_t [v^{(1)}] \leq C_2 C_1
\]
for $t\in [0,T/K]$, where $C_2 \geq 1$ is a constant. Thus, finding the solution $\widetilde{u}^{(2)}$ to \eqref{tildep_eq} on $[T/K,2T/K]$, with ``retarded $v$'', i.e., with $v\coloneqq v^{(1)} (t-T/K)$, will result in an estimate
\[\alpha_t [\widetilde{u}^{(2)}]+ \beta_t [\widetilde{u}^{(2)}] \leq C_3,
\]
where now $C_3$ will depend on both $C_1$ and~$C_2$. It is thus not clear that the constants in the a~priori bounds will not accumulate over $t\in [0,T]$, as the iterations proceed for $k=1,\ldots , K$.  This issue might seem similar to the difficulty described above, but it is in fact completely different. It is not concerned with invoking the local energy inequality to obtain estimates on $\alpha_t [\widetilde{u}^{(k)}]+\beta_t [\widetilde{u}^{(k)}]$ as the iteration proceeds, but rather is concerned with ensuring that the a~priori estimate obtained for the final iterate $\widetilde{u}$ (on time interval $[0,T]$) is a nonlinear one. 

To be precise, since \eqref{tildep_eq} is a linear system, it will result with an a~priori bound with an implicit constant dependent on the advecting~$v$. The Navier-Stokes equations are not linear, and so, even though the a~priori estimate on the desired solution must depend only on $\| u_0\|_M$, the implicit constant cannot depend on the advecting velocity. This means that an iterative construction using the linear system \eqref{tildep_eq} must quantitatively keep track of the implicit constants in the estimates on 
\[\alpha_t [\widetilde{u}^{(k)}]+ \beta_t [\widetilde{u}^{(k)}] 
\]
on $[0,kT/K]$. This is achieved in Section~\ref{sec_pf_main}, where the most important constant is fixed in~\eqref{choice_C0}. \\

The structure of the paper is as follows. After introducing some preliminary concepts in Section~\ref{sec_prelims}, we discuss, in Section~\ref{sec_prop_C}, some properties of any family $\sc$ of balls satisfying the conditions in Definition~\ref{Cr}.
In Section~\ref{sec_ap}, we then prove Theorem~\ref{thm_apriori} , including establishing in Section~\ref{sec_press} the pressure estimate. Section~\ref{sec_constr} is devoted to the proof of Theorem~\ref{thm_main}, which is much more involved in the case of assumption~\eqref{ass_extra}. Namely, in this case we discuss global well-posedness of the linear system~\eqref{v_system_intro} in Section~\ref{sec_lin_sys}, the a~priori estimate of $\div \widetilde{u}$ in $M$ and its Gaussian decay (Section~\ref{sec_lin_decay}), the estimate on the pressure~$\widetilde{p}$ (Section~\ref{sec_lin_p}), the a~priori estimate for~$\widetilde{u}$ (Section~\ref{sec_lin_apr}), the regularization lemma (Section~\ref{sec_reg_lemma}), and we conclude the construction in Section~\ref{sec_pf_main}. The following section
then proves Theorem~\ref{thm_main} in the case of assumption~\eqref{ass2}.

\section{Notation and preliminaries}\label{sec_prelims}

Here we introduce some preliminary concepts. 
We will use  the summation convention over repeated indices.
We denote by $C>0$ a universal constant, whose value may depend on $\sigma $ and $\eta$ and may change from line to line. Except for this, we will use constants $C_0,C_1,C_2\geq 1$ (in~\eqref{v2condition}--\eqref{v3condition}, \eqref{apr} and \eqref{vtilde2}--\eqref{vtilde3},  respectively), whose values \emph{do not change} from line to line.  For a given ball $\QQ\subset \R^3$, denote by $x_{\QQ}$ its center and by $r_{\QQ}$ its radius.

\begin{definition}\label{cubes}
Given $\QQ\in \mathscr{C}$, we set $P^{(0)} =\QQ^{(0)} \coloneqq \QQ$ and, for $n\geq 1$, denote by
\begin{enumerate}
\item[(i)] $P^{(n)}\coloneqq\left\{\QQ'\in\mathscr{C}|\QQ'\cap \QQ^{(n-1)}\neq\emptyset\right\}$ the collection of balls in layers from $0$ until $n$,
\item[(ii)] $\QQ^{(n)} \coloneqq \bigcup_{\QQ'\in P^{(n)}} \QQ'$ the union of balls in layers $0$ to  $n$, and
\item[(iii)] $S^{(n)}  \coloneqq \bigcup_{B\subset \QQ^{(n)},\, B\not \subset \QQ^{(n-1)}} B$ the union of balls in the $n$-th layer.
\end{enumerate}
\end{definition}

For  $B\in \sc$ we will consider the pressure decomposition \eqref{localpress_intro} of the form 
\begin{equation}\label{localpress}
		\begin{split}
		p(x,t)-p_\QQ(t)
			 &=G_{ij}^B (u_i u_j) \coloneqq 
     			 -|u(x,t)|^2 
      			+\text{p.v.}\int_{y\in \QQ^{(7)}}K_{ij}(x-y)(u_i(y,t)u_j(y,t))\d y
      			\\&\indeq
      			+\int_{y\notin \QQ^{(7)}}\left(K_{ij}(x-y)-K_{ij}(x_\QQ-y)\right)(u_i(y,t)u_j(y,t))\d y
			\end{split}
		\end{equation}
for $x\in B^{(3)}$, i.e., we take $B^*\coloneqq B^{(3)}$, $B^{**} \coloneqq B^{(7)}$.

\section{Properties of $\mathscr{C}$}\label{sec_prop_C}

From here on, we fix any $\mathscr{C}$ as in Definition~\ref{Cr}. First, given a ball $\QQ\in \mathscr{C}$,  we establish a lower bound on the distance between two points $x,y$ located in different layers of~$\QQ$.

\cole
\begin{lemma}[Separation of layers]\label{lem_sepa}
Let $\QQ\in \mathscr{C}$ and
 $n\in\mathbb{N}_0$. Suppose that $\QQ_1 = B(x_1,r_1) \in P^{(n)}$ and $\QQ_2=B(x_2,r_2) \in S^{(n+4)}$.  Then,
\eqnb\label{sep}
\mathrm{dist}\,(\QQ_1,\QQ_2)
\geq
\max
\{\kappa (r_1+r_2),1\}
,
\eqne
where $\kappa \coloneqq  \sqrt{1+\fractext{2}{\eta^2}}-1$.
\end{lemma}

\begin{proof}[Proof of Lemma~\ref{lem_sepa}]
We first claim that 
\eqnb\label{123}
S^{(n)} \cap S^{(n+2)} =\emptyset.
\eqne
Indeed, suppose that there exists $x\in S^{(n)} \cap S^{(n+2)}$. Since $S^{(n+2)}$ is a union of balls, this means that $x\in \QQ'$ for some $\QQ'\subset S^{(n+2)}$. But since  $S^{(n)} \subset \QQ^{(n)}$, we have that $x\in \QQ^{(n)}\cap \QQ' $. Therefore, $\QQ'\subset \QQ^{(n+1)}$. This is a contradiction since each ball in $S^{(n+2)}$ is not a subset of~$\QQ^{(n+1)}$.

Let $x\in \QQ_1$, $y\in \QQ_2$ be such that 
\[
|x-y| = \mathrm{dist}\,(\QQ_1,\QQ_2).
\]
Noting that $S^{(n+2)}$ is located between $\QQ^{(n)}$ and $S^{(n+4)}$, there exists $\QQ_3=B(x_3,r_3) \subset S^{(n+2)}$ such that
$[x,y] \cap \QQ_3 \ne \emptyset$. Denote by $z$ the center of the line segment $ [x,y] \cap \QQ_3$.  
\begin{center} 
\includegraphics[width=0.7\textwidth]{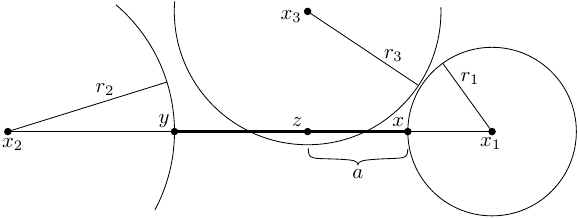}
  \captionof{figure}{A sketch of the proof of~\eqref{pita1}.}\label{fig_pita}
  \end{center} 
Applying the Pitagorean identity  to the triangle with vertices $x_1$, $x_3$, and $z$ gives 
\eqnb\label{pita1}
	(a+r_1)^2 
	= |x_3-x_1|^2 - |x_3-z|^2 
	\geq (r_1+r_3)^2 - r_3^2 
	= r_1^2 +2r_1r_3,
\eqne
where $a\coloneqq |z-x|$ (see Figure~\ref{fig_pita}), using \eqref{123} in the above inequality.  Thus,
\eqnb\label{temp00}
a\geq-r_1+\sqrt{r_1^2+2r_1r_3}
	=r_1\left(-1+\sqrt{1+\frac{2r_3}{r_1}}\right)
	\geq r_1 \kappa .
\eqne
Repeating the same argument for the triangle with vertices $x_2$, $x_3$, $z$, we obtain~\eqref{sep} 
with $\kappa(r_1+r_2)$.
The inequality \eqref{pita1} also implies that $a\geq 1/2$. Indeed, \eqref{pita1} is equivalent to $a^2+2ar_1 \geq 2r_1 r_3$, which implies that either $a^2 \geq r_1r_3$ or $2ar_1 \geq r_1r_3$. Either case implies that $a\geq 1/2$.
Consequently, we obtain
$\mathrm{dist}\,(\QQ_1,\QQ_2)
\geq
1
$.
\end{proof}

\cole
\begin{corollary}
\label{P01}
Let $\QQ\in \mathscr{C}$ and $x\in \QQ$. Then,
\eqnb\label{whereisy}
y\in \QQ^{(4\lceil |x-y|\rceil)},\quad \text{ for each } y \in \R^3.
\eqne
\end{corollary}

\begin{proof}[Proof of Corollary~\ref{P01}]
By \eqref{sep},
we have
$B(x,1) \subset \QQ^{(4)}$, and so, by induction, $B(x,k)\subset \QQ^{(4k)}$ for every $k\geq 1$. Thus, if $|x-y| =d$, where $d\in (k-1,k]$, then $y\in \QQ^{(4k)}$, which shows~\eqref{whereisy}.
\end{proof}
The following verifies the asymptotic estimate of sizes of balls $B\in \sc$, as $|x_B|\to \infty$, stated in~\eqref{asympt}.

\begin{corollary}\label{cor_forintro} There exists a constant $C=C(\sc)>0$ such that $|B|^{1/3} \leq C ( 1+ |x_B|)$ for all $B\in \sc$. 
\end{corollary}

\begin{proof}[Proof of Corollary~\ref{cor_forintro}]
Let $B_0\in \sc$ be such that $0\in B_0$
and $B \in \sc $ such that $|x_B|\geq 2|x_{B_0}|$ and $B_0 \not \subset B^{(4)}$. Then,
\[
|x_B| \geq |x_B - x_{B_0} | - |x_{B_0}| \geq |x_B - x_{B_0} | - |x_{B}|/2, 
\]
so that $|x_B| \geq 2 |x_B - x_{B_0}|$. Therefore, \eqref{sep} implies that
\[
|x_B| > 2 \,\mathrm{dist}(B,B_0) \geq 2 \max \{ \kappa (r_B +r_{B_0} ),1 \} \geq c r_B = c |B|^{1/3},
\]
where, in the last inequality, we used that $r_B \geq 1$ (recall~Definition~\ref{Cr}). If $|x_B|\leq 2|x_{B_0}|$ or $B_0 \subset B^{(4)}$, then $|B|\lec_{\sc} 1$.
\end{proof}

In the next statement, we compare the distance of points in well-separated balls
compared to the distance between their centers.

\cole
\begin{lemma}[Arbitrary points vs.~centers]\label{lem_sep}
Let $\QQ,\QQ'\in\mathscr{C}$ be balls of radii $r_1$ and $r_2$ respectively such that $\QQ'\not\subset \QQ^{(3)}$. Then,
\eqnb\label{EQ01}
(1-\beta ) | x_\QQ - x_{\QQ'} |
\leq |x-y|
\leq 2 | x_\QQ - x_{\QQ'} |
\eqne
for all $x\in \QQ$, $y\in \QQ'$, where $\beta \coloneqq\left( 1+ \frac{2}{\eta^2} \right)^{-1/2} \in (0,1)$.
 \end{lemma}
Note that Lemma~\ref{lem_sep} implies also that
\eqnb\label{EQ01cons}
|x_B - y| \lec_\eta | x-y | \lec_\eta |x_B -y | \qquad \text{ for all }x\in B \in \sc , y\not \in B^{(3)}.
\eqne
Indeed, if $B'\in \sc$ is such that $y\in B'$, then $| x-y| \geq (1-\beta )^{-1} |x_B - x_{B' }| \geq \frac12 (1-\beta )^{-1} |x_B - y |$ (since $x_B \in B$), which gives the first inequality in~\eqref{EQ01cons}. Similarly, $|x_B - y | \geq (1-\beta )|x_B-x_B'| \geq (1-\beta ) |x_B-y|/2$, which gives the second inequality.

\begin{proof}[Proof of Lemma~\ref{lem_sep}.]
Let $x\in \QQ$ and $y\in \QQ'$.
Then we have
\[
 \begin{split}
  |x-y|
  &\leq |x-x_\QQ|
  +|x_\QQ-x_{\QQ'}|
  +|x_{\QQ'}-y|
  \leq 
  |x_\QQ-x_{\QQ'}|
  +r_1
  +r_2
  \leq  2|x_\QQ-x_{\QQ'}|
   ,
  \end{split}
\]
proving the second inequality in~\eqref{EQ01}.

For the first inequality in~\eqref{EQ01},  we first use \eqref{sep} to note that
\eqnb\label{EQ01a}
	|x_\QQ-x_{\QQ'}|= r_1 + r_2 + \mathrm{dist}\, (\QQ,\QQ') \geq r_1 + r_2 + \kappa (r_1+r_2 ) = \sqrt{1+\frac{2}{\eta^2 } } (r_1+r_2) .
	\eqne
Thus,
 \[
   |x_\QQ-x_{\QQ'}|
   \leq 
   |x_\QQ-x|
   +|x-y|
   +|y-x_{\QQ'}|
   \leq 
   |x-y|
   +r_1
   +r_2
   \leq |x-y|
   + \beta |x_\QQ-x_{\QQ'}|,
  \]
concluding the proof of the first inequality in~\eqref{EQ01}.
\end{proof}

 We now mention a property which will be used in the construction of solutions (Section~\ref{sec_constr}) to treat $\div \widetilde{u}$ (in~\eqref{w_ptwise}) and the pressure estimate (in~\eqref{vcond11}, as well as in estimating the last term in~\eqref{app1}) of a solution $\widetilde{u}$ to the linear system~\eqref{v_system_intro}.  
  Namely, we need to quantify some behavior of balls from $\mathscr{C}$ touching a ball of large radius. To this end, let $B_0\in \mathscr{C}$ be any ball such that $0\in B_0$ and for each $R>0$ sufficiently large (depending on $\mathscr{C}$) so that $B_0^{(7)} \subset B(0,R)$, we set
\[
\mathscr{C}_R \coloneqq \{ B \in \mathscr{C} \colon B\cap B(0,R)  \ne \emptyset \}.
\]

\begin{proposition}\label{prop_CR}
For every $B \in \mathscr{C}_R$,
\eqnb\label{vcond1}
| x_B | \lec_\eta R\qquad \text{ and }\qquad r_B \lec_\eta R.
\eqne
\end{proposition}

\begin{proof}[Proof of Proposition~\ref{prop_CR}]  If $|x_B|\leq R$, then the claim follows trivially. If $|x_B |>R$, then $B\not \in B_0^{(7)}$, and so by the same argument as in~\eqref{EQ01a},
\eqnb\label{B0aa}
r_{\QQ} \leq \left( 1+ \frac{2}{\eta^2 } \right)^{-1/2} | x_{\QQ}-  x_{\QQ_0} |.
\eqne
Thus,
\begin{align}
  \begin{split}
| x_{\QQ}-  x_{\QQ_0} |
&\leq | x_{\QQ_0} | +  |x_{\QQ} |  \leq  R + | y | + |x_{\QQ} - y| \leq  2R + r_{\QQ}
\\&
\leq 2R  + \left( 1+ \frac{2}{\eta^2 } \right)^{-1/2} | x_{\QQ}-  x_{\QQ_0} |,
  \end{split}
   \label{EQ05}
   \end{align}
where $y\in B$ is such that $|y|\leq R$ (which exists since $B' \in \mathscr{C}$). Consequently, $| x_{\QQ}-  x_{\QQ_0} | \lec_\eta R $, and \eqref{vcond1} follows by the triangle inequality and~\eqref{B0aa}. 
\end{proof}

\section{A~priori estimates}\label{sec_ap}

Here we prove Theorem~\ref{thm_apriori}.    Given $\QQ\in\sc$, let $\phi_\QQ\in C_c^\infty (B^{(3)})$ be such that  $\phi_\QQ=1$ on $\QQ$ and 
    \begin{equation}\label{phiq}
    	\|D^k \phi_\QQ\|_{L^{\infty}}\lec_k {|\QQ|^{-k/3}}
    \end{equation}
for every $k\geq 0$; note that the choice of such $\phi_B$ is possible due to the separation Lemma~\ref{lem_sepa}. We will use the following pressure estimate, which we prove in Section~\ref{sec_press}.

\begin{lemma}[Pressure estimate]\label{lem_pressure}
Let $u$ be any local energy solution on $[0,T]$ and let $p(x,t) - p_B(t)$ be given by the pressure decomposition~\eqref{localpress}. Then
 \begin{equation}\label{pressureest}
      \begin{split}
          \frac{1}{|\QQ|^{1/3}}\int_0^T\int_{\QQ^{(3)}}|p-p_\QQ(t)|^{3/2}
          &\leq 
   C\sup_{\QQ'\subset \QQ^{(7)}}\frac{1}{|\QQ'|^{1/3}}\int_0^T\int_{\QQ'}|u|^3+
   CT|\QQ|^{2/3}\alpha_T^{3/2}
   .
         \end{split}
  \end{equation}
\end{lemma}
Recall from~\eqref{alphabetas} that $\alpha_T = \esssup_{0<s<t} \| u \|_{M}^2$.\\


\begin{proof}[Proof of Theorem~\ref{thm_apriori}]
To prove \eqref{main_apr}, we fix  $\QQ\in \sc$ and use the   local energy inequality \eqref{localenergy} to obtain
	\begin{equation}\label{ap1}
	\begin{split}
		\frac{1}{2}\int|u(t)|^2\phi_{\QQ}+\int_0^t\int |\nabla u|^2\phi_\QQ
			&\leq
			\frac{1}{2}\int|u(0)|^2\phi_\QQ
			+\frac{1}{2}\int_0^t\int |u|^2\Delta\phi_\QQ
			\\ &\indeq
			+\frac{1}{2}\int_0^t\int(|u|^2+2p)(u\cdot\nabla\phi_\QQ).
	\end{split}
	\end{equation}
Next, we estimate the terms on the right-hand side of~\eqref{ap1}. Using \eqref{phiq} and conditions (i)--(ii) in Definition~\ref{Cr}, we see that
	\begin{equation}\label{ap2}
	\begin{split}
		\frac{1}{2}\int_0^t\int |u|^2\Delta\phi_\QQ
		&\lec_\sigma
		\frac{t}{|\QQ|^{2/3}}\esssup_{0<s<t}\sup_{\QQ'\subset \QQ^{(3)}}\int_{\QQ'}|u(s)|^2 \lec_{\eta}  t \alpha_t .
		\end{split}
	\end{equation}
For the last term on the right-hand side of \eqref{ap1}, we use \eqref{phiq} and the pressure \eqref{pressureest} to obtain
	\begin{equation}\label{ap3}
	\begin{split}
		&\frac{1}{2}\int_0^t\int(|u|^2+2p)(u\cdot\nabla\phi_\QQ)
		\lec
		\frac{1}{|\QQ|^{1/3}}\int_0^t\int_{\QQ^{(3)}}(|u|^3+|p-p_\QQ|^{3/2})
		\\&\indeq
		\lec_{\sigma,\eta }
t|\QQ|^{2/3}\alpha_t^{3/2}
+ \sup_{\QQ'\subset \QQ^{(7)}}\frac{1}{|\QQ'|^{1/3}}\int_0^t\int_{\QQ'}|u|^3
.
		\end{split}
	\end{equation}
Applying the Gagliardo-Nirenberg inequality for the ball $\QQ'$, we get
	\[
		\int_{\QQ'}|u|^3
		\lec \left(\int_{\QQ'}|u|^2\right)^{3/4}\left(\int_{\QQ'}|\nabla u|^2\right)^{3/4}+\frac{1}{|\QQ'|^{1/2}}\left(\int_{\QQ'}|u|^2\right)^{3/2}
	\]
for all $t\in[0,T]$.	
Thus,
   	 \begin{equation}\label{ap4}
   	 \begin{split}
    		\frac{1}{|\QQ'|^{1/3}}\int_0^t\int_{\QQ'}|u|^3
		&\lec  t^{1/4} |\QQ'|^{2/3}\left(\frac{1}{|\QQ'|^{2/3}}\esssup_{0<s<t}\int_{\QQ'}|u(s)|^2\right)^{3/4}\left(\frac{1}{|\QQ'|^{2/3}}\int_0^t\int_{\QQ'}|\nabla u|^2\right)^{3/4}
		\\&\indeq
		+t |\QQ'|^{1/6}\left(\frac{1}{|\QQ'|^{2/3}}\esssup_{0<s<t}\int_{\QQ'}|u(s)|^2\right)^{3/2}
		\\&\lec  t^{1/4}  |\QQ'|^{2/3}(\alpha_t+\beta_t)^{3/2}
  	  \end{split}
  	  \end{equation}
for every $\QQ' \in \mathscr{C}$ and $t\in [0,1]$,
where the constant depends on~$T$.
Applying this in~\eqref{ap3} and using the resulting inequality and  \eqref{ap2} in the local energy inequality \eqref{ap1}, we get
   	 \[
  	  \begin{split}
    		&
\frac12
              \int_{\QQ}|u(t)|^2+\int_0^t\int_{\QQ} |\nabla u|^2
		\\&\indeq
		\leq  \frac12 \int |u_0|^2 \phi_\QQ +	Ct\alpha_t+Ct|\QQ|^{2/3}\alpha_t^{3/2}+C (\alpha_t+\beta_t)^{3/2}\sup_{\QQ'\subset \QQ^{(3)}}t^{1/4}|\QQ'|^{2/3}
		.
		\end{split}
   	 \]
 Dividing by $|\QQ|^{2/3}$ and recalling Definition~\ref{Cr}(ii), we see that
  	  \begin{equation}
    	 \begin{split}
&    		\frac{1}{|\QQ|^{2/3}}\int_{\QQ}|u(t)|^2
    +\frac{1}{|\QQ|^{2/3}}\int_0^t\int_{\QQ} |\nabla u(x,s)|^2
		\\ &\indeq
		\lec_{\eta,\sigma } \| u_0 \|_{M}^2+ t\alpha_t+t\alpha_t^{3/2}+ t^{1/4}  (\alpha_t+\beta_t)^{3/2}.
	 \end{split}
   \label{EQ03}
	 \end{equation}
Therefore, $\alpha_t+\beta_t\lec_{\sigma,\eta } \| u (0) \|_{M}^2 + t(\alpha_t+\beta_t)+ t^{1/4}  (\alpha_t+\beta_t)^{3/2}$.

Now, we apply the barrier argument to~\eqref{EQ03}.
Since  $\alpha_t+\beta_t$ is continuous, by assumption, there exists $T>0$ such that
	\[
		\alpha_t+\beta_t\leq C\|u_0\|^2_{M}
		\comma
		[0,T)
		,
	\]
and $T$ is maximal, i.e., $\alpha_T+\beta_T=C\|u_0\|^2_{M}$,
where $C$ is sufficiently large and fixed until the end of the proof. Let
	\[
	T_0=\min\{(2C)^{-1},(4C\alpha_0^{1/2})^{-4}\}.
	\]
	If $T<T_0$, then $\alpha_T+\beta_T<C\alpha_0$, which is not possible. Therefore, $T_0\leq T$ and
	\[
		\sup_{0<t<T}\|u(t)\|_{M}^2+\sup_{\QQ\in\sc}\frac{1}{\QQ^{2/3}}\int_0^T\int_{\QQ}|\nabla u|^2\,dx\,dt
		\leq 
		2\|u_0\|_{M}^2,
	\]
concluding the proof of Theorem~\ref{thm_apriori}.
\end{proof}

\subsection{The pressure estimate}\label{sec_press}
Here we prove Lemma~\ref{lem_pressure}, that is, for a given $B\in \sc$, we show that
    \[ \begin{split}
          \frac{1}{|\QQ|^{1/3}}\int_0^T\int_{\QQ^{(3)}}|p-p_\QQ(t)|^{3/2}
          &\lec \sup_{\QQ'\subset \QQ^{(7)}}\frac{1}{|\QQ'|^{1/3}}\int_0^T\int_{\QQ'}|u|^3+
   CT|\QQ|^{2/3}\alpha_T^{3/2}
         \end{split}\]
whenever $p$ satisfies the pressure decomposition \eqref{localpress} for some $u \in M$.

\begin{proof}[Proof of Lemma~\ref{lem_pressure}]
First, recall the pressure decomposition \eqref{localpress},
\begin{equation}\label{localpress_repeat}
		\begin{split}
		p(x,t)-p_\QQ(t)
			 &=
     			 -|u(x,t)|^2 
      			+\text{p.v.}\int_{y\in \QQ^{(7)}}K_{ij}(x-y)(u_i(y,t)u_j(y,t))\d y 
      			\\&\indeq
      			+\int_{y\notin \QQ^{(7)}}\left(K_{ij}(x-y)-K_{ij}(x_\QQ-y)\right)(u_i(y,t)u_j(y,t))\d y\\
      			&=: \In(x,t) + \If(x,t).
			\end{split}
		\end{equation}
From the  Calder\'on-Zygmund inequality, we have
          \begin{equation}\label{In_bound}
             \frac{1}{|\QQ|^{1/3}}\int_0^T\int_{\QQ^{(3)}}|\In|^{3/2}
    \leq 
   \frac{C}{|\QQ|^{1/3}}\int_0^T\int_{\QQ^{(7)}}|u|^{3}.
          \end{equation}
For each $\QQ'\in P^{(7)}$, Definition~\ref{Cr}(i) gives that $|\QQ|^{1/3}   	 	\sim_{\eta }|\QQ'|^{1/3}$,  and so, since there are at most $\sigma^7$ cubes in  $P^{(7)}$, we deduce that
         \begin{equation}\label{inear}
          \begin{split}
            \frac{1}{|\QQ|^{1/3}}\int_0^T\int_{\QQ^{(7)}}|u|^{3}
     &\lec_\sigma 
     \sup_{\QQ'\in P^{(7)}}\frac{|\QQ'|^{1/3}}{|\QQ|^{1/3}}\frac{1}{|\QQ'|^{1/3}}\int_0^T\int_{\QQ'}|u|^{3}\lec_\eta \sup_{\QQ'\subset \QQ^{(7)}}\frac{1}{|\QQ'|^{1/3}}\int_0^T\int_{\QQ'}|u|^{3}.
           \end{split}
         \end{equation}
For $\If$ we observe that  for $y\notin \QQ^{(7)}$ and $x\in \QQ^{(3)}$ we have
         \begin{equation}\label{36}
           | K_{ij}(x-y)-K_{ij}(x_\QQ-y)|
    \leq 
   \frac{C|\QQ|^{1/3}}{|x-y|^4}.
         \end{equation}
(Recall that $K_{ij} (z) = \p_i \p_j (4\pi |z|)^{-1}$.)  Thus, using  \eqref{EQ01},
          \[
             \begin{split}
                 |\If(x,t)|&\lec |\QQ|^{1/3}\int_{\mathbb{R}^3\backslash \QQ^{(7)}}\frac{1}{|x-y|^4}|u(y,t)|^2\d y\\
                 &\lec |\QQ|^{1/3} \sum_{\QQ'\in \mathscr{C} \setminus P^{(7)}}\int_{\QQ'}\frac{1}{|x_\QQ-x_{\QQ'}|^4}|u(y,t)|^2\d y\\
                 &\lec
    		|\QQ|^{1/3} \alpha_T \sum_{\QQ'\in \mathscr{C} \setminus P^{(7)}}\frac{|\QQ'|^{q/3}}{|x_{\QQ}-x_{\QQ'}|^{4}}
   	      \end{split}
         \]
for all $x\in \QQ^{(3)}$. We show below that 
  \begin{align}
  \begin{split}
   \sum_{\QQ'\in \sc \setminus P^{(7)} }\frac{|\QQ'|^{q/3}}{|x_{\QQ}-x_{\QQ'}|^{4}}
   \lec
   |\QQ|^{-1/3}
   .
  \end{split}
   \label{EQ04}
  \end{align}
Given \eqref{EQ04}, we have that
\eqnb\label{EQ04aa}
|\If (x,t) | \lec \alpha_T \eqne
for all $x\in \QQ^{(3)}$, and so the claim \eqref{pressureest} follows by raising both sides of this inequality to the power~$3$ and integrating over $\QQ^{(3)}\times (0,T)$, as well as recalling~\eqref{In_bound}.

It remains to show~\eqref{EQ04}. To this end, denote by $\mathscr{D}$ the collection of cubes $R$ such that 
\eqnb\label{size_of_R}
| R |^{1/3} \in (1/200, 1/100]
\eqne
for every $R\in \mathscr{D}$, defined as follows.
 \begin{enumerate}
 \item Let $A_0$  be the largest cube that can be inscribed in  $\QQ$, as shown in Figure~\ref{fig_D}.
 \item For each $n\geq 1$, denote by $A_n$ the collection of cubes $R$ satisfying \eqref{size_of_R} which form a single layer of cubes of the same size around $T^{(n-1)}\coloneqq \bigcup_{k=0}^{n-1} \bigcup_{R'\in A_{k}} R'$, as sketched in Figure~\ref{fig_D}. Note that then $T^{(n)}=\bigcup_{k=0}^{n} \bigcup_{R'\in A_{k}} R'$ is a cube, so that each $A_n$ is well-defined.
 \item Let $\mathscr{D} \coloneqq \bigcup_{n\geq 1} A_n$.
 \end{enumerate}

\begin{figure}
\tikzset{every picture/.style={line width=0.75pt}} 

\begin{tikzpicture}[x=0.75pt,y=0.75pt,yscale=-0.75,xscale=0.75]

\draw  [fill={rgb, 255:red, 126; green, 211; blue, 33 }  ,fill opacity=1 ] (173,139.5) -- (227,139.5) -- (227,193.5) -- (173,193.5) -- cycle ;
\draw  [fill={rgb, 255:red, 248; green, 231; blue, 28 }  ,fill opacity=1 ] (146,139.5) -- (173,139.5) -- (173,166.5) -- (146,166.5) -- cycle ;
\draw  [fill={rgb, 255:red, 248; green, 231; blue, 28 }  ,fill opacity=1 ] (146,166.5) -- (173,166.5) -- (173,193.5) -- (146,193.5) -- cycle ;
\draw  [fill={rgb, 255:red, 248; green, 231; blue, 28 }  ,fill opacity=1 ] (173,112.5) -- (200,112.5) -- (200,139.5) -- (173,139.5) -- cycle ;
\draw  [fill={rgb, 255:red, 248; green, 231; blue, 28 }  ,fill opacity=1 ] (146,112.5) -- (173,112.5) -- (173,139.5) -- (146,139.5) -- cycle ;
\draw  [fill={rgb, 255:red, 248; green, 231; blue, 28 }  ,fill opacity=1 ] (200,112.5) -- (227,112.5) -- (227,139.5) -- (200,139.5) -- cycle ;
\draw  [fill={rgb, 255:red, 248; green, 231; blue, 28 }  ,fill opacity=1 ] (227,112.5) -- (254,112.5) -- (254,139.5) -- (227,139.5) -- cycle ;
\draw  [fill={rgb, 255:red, 248; green, 231; blue, 28 }  ,fill opacity=1 ] (227,139.5) -- (254,139.5) -- (254,166.5) -- (227,166.5) -- cycle ;
\draw  [fill={rgb, 255:red, 248; green, 231; blue, 28 }  ,fill opacity=1 ] (227,166.5) -- (254,166.5) -- (254,193.5) -- (227,193.5) -- cycle ;
\draw  [fill={rgb, 255:red, 248; green, 231; blue, 28 }  ,fill opacity=1 ] (173,193.5) -- (200,193.5) -- (200,220.5) -- (173,220.5) -- cycle ;
\draw  [fill={rgb, 255:red, 248; green, 231; blue, 28 }  ,fill opacity=1 ] (200,193.5) -- (227,193.5) -- (227,220.5) -- (200,220.5) -- cycle ;
\draw  [fill={rgb, 255:red, 248; green, 231; blue, 28 }  ,fill opacity=1 ] (227,193.5) -- (254,193.5) -- (254,220.5) -- (227,220.5) -- cycle ;
\draw  [fill={rgb, 255:red, 248; green, 231; blue, 28 }  ,fill opacity=1 ] (146,193.5) -- (173,193.5) -- (173,220.5) -- (146,220.5) -- cycle ;
\draw  [fill={rgb, 255:red, 245; green, 166; blue, 35 }  ,fill opacity=1 ] (119,139.5) -- (146,139.5) -- (146,166.5) -- (119,166.5) -- cycle ;
\draw  [fill={rgb, 255:red, 245; green, 166; blue, 35 }  ,fill opacity=1 ] (119,112.5) -- (146,112.5) -- (146,139.5) -- (119,139.5) -- cycle ;
\draw  [fill={rgb, 255:red, 245; green, 166; blue, 35 }  ,fill opacity=1 ] (119,85.5) -- (146,85.5) -- (146,112.5) -- (119,112.5) -- cycle ;
\draw  [fill={rgb, 255:red, 245; green, 166; blue, 35 }  ,fill opacity=1 ] (146,85.5) -- (173,85.5) -- (173,112.5) -- (146,112.5) -- cycle ;
\draw  [fill={rgb, 255:red, 245; green, 166; blue, 35 }  ,fill opacity=1 ] (173,85.5) -- (200,85.5) -- (200,112.5) -- (173,112.5) -- cycle ;
\draw  [fill={rgb, 255:red, 245; green, 166; blue, 35 }  ,fill opacity=1 ] (200,85.5) -- (227,85.5) -- (227,112.5) -- (200,112.5) -- cycle ;
\draw  [fill={rgb, 255:red, 245; green, 166; blue, 35 }  ,fill opacity=1 ] (227,85.5) -- (254,85.5) -- (254,112.5) -- (227,112.5) -- cycle ;
\draw  [fill={rgb, 255:red, 245; green, 166; blue, 35 }  ,fill opacity=1 ] (254,85.5) -- (281,85.5) -- (281,112.5) -- (254,112.5) -- cycle ;
\draw  [fill={rgb, 255:red, 245; green, 166; blue, 35 }  ,fill opacity=1 ] (254,112.5) -- (281,112.5) -- (281,139.5) -- (254,139.5) -- cycle ;
\draw  [fill={rgb, 255:red, 245; green, 166; blue, 35 }  ,fill opacity=1 ] (254,139.5) -- (281,139.5) -- (281,166.5) -- (254,166.5) -- cycle ;
\draw  [fill={rgb, 255:red, 245; green, 166; blue, 35 }  ,fill opacity=1 ] (254,166.5) -- (281,166.5) -- (281,193.5) -- (254,193.5) -- cycle ;
\draw  [fill={rgb, 255:red, 245; green, 166; blue, 35 }  ,fill opacity=1 ] (254,193.5) -- (281,193.5) -- (281,220.5) -- (254,220.5) -- cycle ;
\draw  [fill={rgb, 255:red, 245; green, 166; blue, 35 }  ,fill opacity=1 ] (254,220.5) -- (281,220.5) -- (281,247.5) -- (254,247.5) -- cycle ;
\draw  [fill={rgb, 255:red, 245; green, 166; blue, 35 }  ,fill opacity=1 ] (227,220.5) -- (254,220.5) -- (254,247.5) -- (227,247.5) -- cycle ;
\draw  [fill={rgb, 255:red, 245; green, 166; blue, 35 }  ,fill opacity=1 ] (200,220.5) -- (227,220.5) -- (227,247.5) -- (200,247.5) -- cycle ;
\draw  [fill={rgb, 255:red, 245; green, 166; blue, 35 }  ,fill opacity=1 ] (173,220.5) -- (200,220.5) -- (200,247.5) -- (173,247.5) -- cycle ;
\draw  [fill={rgb, 255:red, 245; green, 166; blue, 35 }  ,fill opacity=1 ] (146,220.5) -- (173,220.5) -- (173,247.5) -- (146,247.5) -- cycle ;
\draw  [fill={rgb, 255:red, 245; green, 166; blue, 35 }  ,fill opacity=1 ] (119,220.5) -- (146,220.5) -- (146,247.5) -- (119,247.5) -- cycle ;
\draw  [fill={rgb, 255:red, 245; green, 166; blue, 35 }  ,fill opacity=1 ] (119,166.5) -- (146,166.5) -- (146,193.5) -- (119,193.5) -- cycle ;
\draw  [fill={rgb, 255:red, 245; green, 166; blue, 35 }  ,fill opacity=1 ] (119,193.5) -- (146,193.5) -- (146,220.5) -- (119,220.5) -- cycle ;
\draw   (162.75,166.5) .. controls (162.75,145.93) and (179.43,129.25) .. (200,129.25) .. controls (220.57,129.25) and (237.25,145.93) .. (237.25,166.5) .. controls (237.25,187.07) and (220.57,203.75) .. (200,203.75) .. controls (179.43,203.75) and (162.75,187.07) .. (162.75,166.5) -- cycle ;
\draw    (506,98.5) -- (269.5,99) ;
\draw [shift={(267.5,99)}, rotate = 359.88] [color={rgb, 255:red, 0; green, 0; blue, 0 }  ][line width=0.75]    (10.93,-3.29) .. controls (6.95,-1.4) and (3.31,-0.3) .. (0,0) .. controls (3.31,0.3) and (6.95,1.4) .. (10.93,3.29)   ;
\draw    (358,207.5) -- (242.5,207.01) ;
\draw [shift={(240.5,207)}, rotate = 0.24] [color={rgb, 255:red, 0; green, 0; blue, 0 }  ][line width=0.75]    (10.93,-3.29) .. controls (6.95,-1.4) and (3.31,-0.3) .. (0,0) .. controls (3.31,0.3) and (6.95,1.4) .. (10.93,3.29)   ;
\draw    (50.75,156) -- (162.5,155.02) ;
\draw [shift={(164.5,155)}, rotate = 179.5] [color={rgb, 255:red, 0; green, 0; blue, 0 }  ][line width=0.75]    (10.93,-3.29) .. controls (6.95,-1.4) and (3.31,-0.3) .. (0,0) .. controls (3.31,0.3) and (6.95,1.4) .. (10.93,3.29)   ;
\draw    (72,180.5) -- (171,180.01) ;
\draw [shift={(173,180)}, rotate = 179.72] [color={rgb, 255:red, 0; green, 0; blue, 0 }  ][line width=0.75]    (10.93,-3.29) .. controls (6.95,-1.4) and (3.31,-0.3) .. (0,0) .. controls (3.31,0.3) and (6.95,1.4) .. (10.93,3.29)   ;
\draw  [fill={rgb, 255:red, 245; green, 166; blue, 35 }  ,fill opacity=1 ] (509,66.5) -- (536,66.5) -- (536,93.5) -- (509,93.5) -- cycle ;
\draw  [fill={rgb, 255:red, 245; green, 166; blue, 35 }  ,fill opacity=1 ] (509,39.5) -- (536,39.5) -- (536,66.5) -- (509,66.5) -- cycle ;
\draw  [fill={rgb, 255:red, 245; green, 166; blue, 35 }  ,fill opacity=1 ] (509,12.5) -- (536,12.5) -- (536,39.5) -- (509,39.5) -- cycle ;
\draw  [fill={rgb, 255:red, 245; green, 166; blue, 35 }  ,fill opacity=1 ] (536,12.5) -- (563,12.5) -- (563,39.5) -- (536,39.5) -- cycle ;
\draw  [fill={rgb, 255:red, 245; green, 166; blue, 35 }  ,fill opacity=1 ] (563,12.5) -- (590,12.5) -- (590,39.5) -- (563,39.5) -- cycle ;
\draw  [fill={rgb, 255:red, 245; green, 166; blue, 35 }  ,fill opacity=1 ] (590,12.5) -- (617,12.5) -- (617,39.5) -- (590,39.5) -- cycle ;
\draw  [fill={rgb, 255:red, 245; green, 166; blue, 35 }  ,fill opacity=1 ] (617,12.5) -- (644,12.5) -- (644,39.5) -- (617,39.5) -- cycle ;
\draw  [fill={rgb, 255:red, 245; green, 166; blue, 35 }  ,fill opacity=1 ] (644,12.5) -- (671,12.5) -- (671,39.5) -- (644,39.5) -- cycle ;
\draw  [fill={rgb, 255:red, 245; green, 166; blue, 35 }  ,fill opacity=1 ] (644,39.5) -- (671,39.5) -- (671,66.5) -- (644,66.5) -- cycle ;
\draw  [fill={rgb, 255:red, 245; green, 166; blue, 35 }  ,fill opacity=1 ] (644,66.5) -- (671,66.5) -- (671,93.5) -- (644,93.5) -- cycle ;
\draw  [fill={rgb, 255:red, 245; green, 166; blue, 35 }  ,fill opacity=1 ] (644,93.5) -- (671,93.5) -- (671,120.5) -- (644,120.5) -- cycle ;
\draw  [fill={rgb, 255:red, 245; green, 166; blue, 35 }  ,fill opacity=1 ] (644,120.5) -- (671,120.5) -- (671,147.5) -- (644,147.5) -- cycle ;
\draw  [fill={rgb, 255:red, 245; green, 166; blue, 35 }  ,fill opacity=1 ] (644,147.5) -- (671,147.5) -- (671,174.5) -- (644,174.5) -- cycle ;
\draw  [fill={rgb, 255:red, 245; green, 166; blue, 35 }  ,fill opacity=1 ] (617,147.5) -- (644,147.5) -- (644,174.5) -- (617,174.5) -- cycle ;
\draw  [fill={rgb, 255:red, 245; green, 166; blue, 35 }  ,fill opacity=1 ] (590,147.5) -- (617,147.5) -- (617,174.5) -- (590,174.5) -- cycle ;
\draw  [fill={rgb, 255:red, 245; green, 166; blue, 35 }  ,fill opacity=1 ] (563,147.5) -- (590,147.5) -- (590,174.5) -- (563,174.5) -- cycle ;
\draw  [fill={rgb, 255:red, 245; green, 166; blue, 35 }  ,fill opacity=1 ] (536,147.5) -- (563,147.5) -- (563,174.5) -- (536,174.5) -- cycle ;
\draw  [fill={rgb, 255:red, 245; green, 166; blue, 35 }  ,fill opacity=1 ] (509,147.5) -- (536,147.5) -- (536,174.5) -- (509,174.5) -- cycle ;
\draw  [fill={rgb, 255:red, 245; green, 166; blue, 35 }  ,fill opacity=1 ] (509,93.5) -- (536,93.5) -- (536,120.5) -- (509,120.5) -- cycle ;
\draw  [fill={rgb, 255:red, 245; green, 166; blue, 35 }  ,fill opacity=1 ] (509,120.5) -- (536,120.5) -- (536,147.5) -- (509,147.5) -- cycle ;
\draw  [fill={rgb, 255:red, 248; green, 231; blue, 28 }  ,fill opacity=1 ] (360,195.5) -- (387,195.5) -- (387,222.5) -- (360,222.5) -- cycle ;
\draw  [fill={rgb, 255:red, 248; green, 231; blue, 28 }  ,fill opacity=1 ] (360,222.5) -- (387,222.5) -- (387,249.5) -- (360,249.5) -- cycle ;
\draw  [fill={rgb, 255:red, 248; green, 231; blue, 28 }  ,fill opacity=1 ] (387,168.5) -- (414,168.5) -- (414,195.5) -- (387,195.5) -- cycle ;
\draw  [fill={rgb, 255:red, 248; green, 231; blue, 28 }  ,fill opacity=1 ] (360,168.5) -- (387,168.5) -- (387,195.5) -- (360,195.5) -- cycle ;
\draw  [fill={rgb, 255:red, 248; green, 231; blue, 28 }  ,fill opacity=1 ] (414,168.5) -- (441,168.5) -- (441,195.5) -- (414,195.5) -- cycle ;
\draw  [fill={rgb, 255:red, 248; green, 231; blue, 28 }  ,fill opacity=1 ] (441,168.5) -- (468,168.5) -- (468,195.5) -- (441,195.5) -- cycle ;
\draw  [fill={rgb, 255:red, 248; green, 231; blue, 28 }  ,fill opacity=1 ] (441,195.5) -- (468,195.5) -- (468,222.5) -- (441,222.5) -- cycle ;
\draw  [fill={rgb, 255:red, 248; green, 231; blue, 28 }  ,fill opacity=1 ] (441,222.5) -- (468,222.5) -- (468,249.5) -- (441,249.5) -- cycle ;
\draw  [fill={rgb, 255:red, 248; green, 231; blue, 28 }  ,fill opacity=1 ] (387,249.5) -- (414,249.5) -- (414,276.5) -- (387,276.5) -- cycle ;
\draw  [fill={rgb, 255:red, 248; green, 231; blue, 28 }  ,fill opacity=1 ] (414,249.5) -- (441,249.5) -- (441,276.5) -- (414,276.5) -- cycle ;
\draw  [fill={rgb, 255:red, 248; green, 231; blue, 28 }  ,fill opacity=1 ] (441,249.5) -- (468,249.5) -- (468,276.5) -- (441,276.5) -- cycle ;
\draw  [fill={rgb, 255:red, 248; green, 231; blue, 28 }  ,fill opacity=1 ] (360,249.5) -- (387,249.5) -- (387,276.5) -- (360,276.5) -- cycle ;

\draw (309,186.4) node [anchor=north west][inner sep=0.75pt]    {$A_{1}$};
\draw (380,77.4) node [anchor=north west][inner sep=0.75pt]    {$A_{2}$};
\draw (31,146.4) node [anchor=north west][inner sep=0.75pt]    {$\QQ$};
\draw (50,174.4) node [anchor=north west][inner sep=0.75pt]    {$A_{0}$};

\end{tikzpicture}
\captionsetup{width=.8\linewidth}
\captionof{figure}{A sketch of the construction of the set $\mathscr{D} = \cup_{n\geq 1} A_n$. Note that $\mathscr{D}$ is a family of cubes for a given $B\in \sc$.}\label{fig_D}
\end{figure}
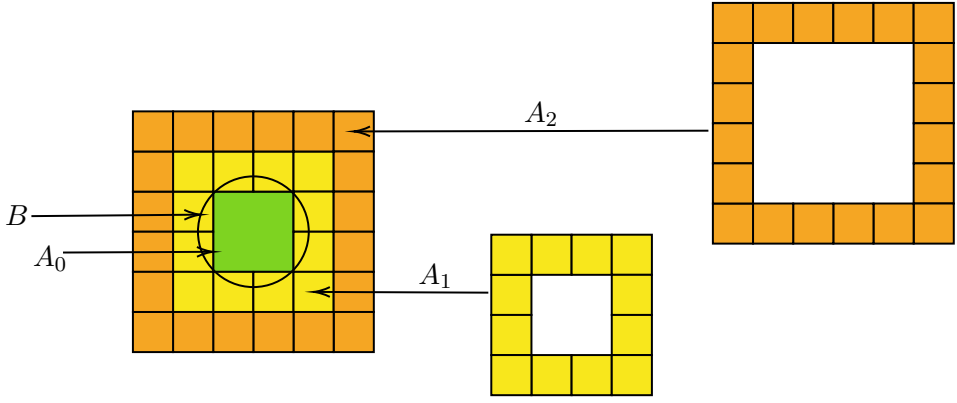

  We emphasize that $\mathscr{D}$ is chosen for a given $B\in \sc$. Note that 
 \eqnb\label{count_of_An}
 \card A_n \lec (|\QQ|^{1/3} + n )^2
 \eqne
for each $n\geq  1$, which is clear from the construction. Additionally, $\mathbb{R}^3\setminus \QQ^{(3)}$ is covered by cubes in~$\mathscr{D}$,
 as $\mathscr{D}$ covers $\mathbb{R}^{3}\setminus A_0$.

 Now, consider any $\QQ'\in \mathscr{C}$ such that $\QQ' \not\subset \QQ^{(3)}$, and denote  the intersection point of the line segment $[x_\QQ,x_{\QQ'}]$ with the sphere $\p \QQ'$ by $y$, denote the center of the line segment $[y,x_{\QQ'}]$ by $z$, denote the radius of $\QQ'$ by  $r' = 2|z-x_{\QQ'}|=2|z-y|\geq 1$, and let  \[
 \mathcal{E}_{\QQ'} \coloneqq \{ R\in \mathscr{D} \colon R \cap B(z,r'/4) \ne \emptyset \};
 \] 
see Figure~\ref{fig_cover} for a sketch. 
 \begin{center} 
\includegraphics[width=0.7\textwidth]{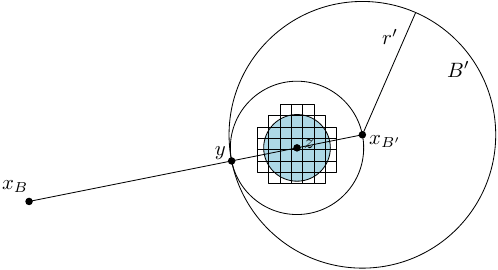}
\captionsetup{width=.8\linewidth}
  \captionof{figure}{A sketch of the collection $\mathcal{E}_{\QQ'}$. Note that $\mathcal{E}_{\QQ'}$ is a family of cubes for a given pair of $B,B'\in \sc$, such that $B'\not \subset B^{(3)}$.}\label{fig_cover}
  \end{center} 
Note that 
 \begin{equation}\label{rset}
 	\bigcup_{R\in \mathcal{E}_{\QQ'} } R \subset B(z,r'/2),
 \end{equation}
 which is clear from the sketch in Figure~\ref{fig_cover} and can be verified by writing
 \[
 	|b- z| \leq |b-c | + |c-z| 
	\leq \sqrt{3} |R|^{1/3} + r'/4 
	\leq \frac1{50} + r'/4
	 \leq r'/2
 \]
 for every $b\in R \in \mathcal{E}_{\QQ'}$, where $c$ denotes any point in $R\cap B(z,r'/4)$. In particular,
 \eqnb\label{010}
 	|x_\QQ - x_{\QQ'}| 
	\geq |x_\QQ - x_R | 
	\comma \text{ for every } R\in \mathcal{E}_{\QQ'}.
 \eqne
 Thus, since $B(z,r'/4)$ is four times smaller than $\QQ'$, and $B(z,r'/4)$ is covered by cubes from $\mathcal{E}_{\QQ'}$, we have $|\QQ'|\lec (r')^3\lec \sum_{R\in \mathcal{E}_{\QQ'}}1$.
Hence,
 \begin{equation}\label{cubetrans}
 \begin{split}
 \sum_{\QQ'\not\subset \QQ^{(3)}}\frac{|\QQ'|^{1/3}}{|x_\QQ - x_{\QQ'}|^4} 
 &\lesssim  
 \sum_{\QQ' \not\subset \QQ^{(3)}} \sum_{R\in \mathcal{E}_{\QQ'}} \frac{1}{|x_\QQ - x_{R}|^4} 
 \lesssim 
 \sigma \sum_{R\in \mathscr{D}} \frac{1}{|x_\QQ - x_{R}|^4} \\
 &\lec 
 \sigma \sum_{n\geq 1} \sum_{R\in A_n}  \frac{1}{\left( |\QQ|^{1/3} + n \right)^4 }\lec \sigma \sum_{n\geq 1}  \frac{1}{\left( |\QQ|^{1/3} + n \right)^2 } \lec \sigma |\QQ|^{-1/3},
\end{split} 
 \end{equation}
 proving \eqref{EQ04}, as required, where we used \eqref{010} in the first inequality; we recalled that each $\QQ' \in \sc$ intersects at most $\sigma$ other cubes (so that we will count at most $\sigma$ overlaps when counting cubes from all families $\mathcal{E}_{\QQ'}$, $\QQ'\in \sc$) in the second inequality. In the third inequality, we observed that 
 \[
  |x_\QQ-x_R|
   \geq
   \frac{|\QQ|^{1/3}}{2}
   +\left(\frac{n}{2}-\frac{1}{4}\right)|R|^{1/3}
   \gec |\QQ|^{1/3}+n
 \]
 for any $R\in A_n$, $n\geq 1$, and we used \eqref{count_of_An} in the fourth inequality.  For the last inequality  in~\eqref{cubetrans}, we used  the elementary fact 
 \eqnb\label{fact}
\sum_{n\geq 1} (a+n)^{-2} \lec a^{-1}
\eqne
for any $a\geq 1$,  which follows by comparing the left-hand side with $\int_a^\infty x^{-2}\,\d x$.
\end{proof}

\section{Construction of solutions}\label{sec_constr}
  Having derived the a~priori estimate \eqref{apr}, we now prove Theorem~\ref{thm_main}. 
  
We first consider the case of the assumption \eqref{ass_extra}. In order to construct the local energy solution $u$, we will use a non-divergence-free linear system, which we introduce in Section~\ref{sec_lin_sys} below. We will then address decay estimates on the divergence of its solution (Section~\ref{sec_lin_decay}), its pressure estimate (Section~\ref{sec_lin_p}), and an estimate of $\widetilde{u}$ in $M$  (Section~\ref{sec_lin_apr}). In Section~\ref{sec_pf_main}, we will apply it to conclude the proof of Theorem~\ref{thm_main} under the assumption~\eqref{ass_extra}.
Then, in Section~\ref{sec_ass2}, we prove Theorem~\ref{thm_main} under the assumption~\eqref{ass2}.

\subsection{A linear system with divergence}\label{sec_lin_sys}
We first address the well-posedness of the system
\begin{equation}\label{v_system}
\begin{split}
	\p_t  \widetilde{u}+(v\cdot\nabla)\widetilde{u}-\Delta \widetilde{u}+\nabla \widetilde{p}&=0,
	\\
   	\widetilde{u}(0)=\widetilde{u}_0,
\end{split}
\end{equation} 
where $v$, $\widetilde{u}_0$ are given and have compact supports in space. We note that the functions $v$ and $\widetilde{u}_0$ will be obtained
by means of cut-offs, and are thus \emph{not}~divergence-free. We define the pressure function
\eqnb\label{pressure_pn}
\begin{split}
	\widetilde{p}(x)
	&\coloneqq 
	-\frac{1}{3}\delta_{ij}(v_i(x)\widetilde{u}_j(x))
	+\text{p.v.}\int_{ \mathbb{R}^3}K_{ij}(x-y)(v_i(y)\widetilde{u}_j(y))\,\d y
	\\ & \indeq
	-\int_{ \mathbb{R}^3}L_{j}(x-y)\left(\left(\div v(y)\right)\widetilde{u}_j(y)\right)\,\d y
\end{split}
\eqne
for each $t$, which is omitted in the notation; above, $K_{ij}(y)=\partial_i\partial_j(4\pi|y|)^{-1}$ was defined in~\eqref{localpress}, and  
\[
	 L_{j}(y)\coloneqq \frac{1}{\omega_3}\frac{y_j}{|y|^3},
\]
denoting \(\displaystyle \omega_d \coloneqq |S^{d-1}| = \frac{2\pi^{d/2}}{\Gamma(d/2)}\). Thus, $\widetilde{p}$ satisfies
\begin{equation}\label{approxpress} 
	\Delta \widetilde{p}=\div\left((v\cdot\nabla) \widetilde{u}\right)
\end{equation} 
in~$\R^3$. Thanks to \eqref{approxpress}, $w \coloneqq \div\widetilde{u}$ obeys the heat equation,
\begin{equation}\label{div}
\begin{split}
	\p_t  w -\Delta w&=0,\\
   	 w(0)&=\div \widetilde{u}_0.
\end{split}
\end{equation}
We now sketch the existence and uniqueness of a solution to~\eqref{v_system} with $\widetilde p$ defined by~\eqref{pressure_pn}.

\begin{proposition}[Well-posedness in the energy class]
\label{prop:lin-v}
Let $T>0$, $\widetilde{u}_0\in L^2(\mathbb{R}^3)$, and
$v\in L^\infty(0,T;W^{1,\infty}(\mathbb{R}^3))$. Then, there exists a unique solution 
  \begin{equation}\label{XT_def}
   \widetilde{u}\in X_T\coloneqq C ([0,T];L^2)\cap L^2(0,T;H^1)   
  \end{equation}
to \eqref{v_system}
with \eqref{pressure_pn}
such that 
\eqnb\label{tildeu_L2est}
\|\widetilde{u}(t)\|_2^2+\int_0^t \|\nabla \widetilde{u}(s)\|_2^2\,\d s
\;\le\;\,\ee^{C t}\,\|\widetilde{u}_0\|_2^2
\eqne
for all $t\in [0,T]$, where $C$ depends only on $\|v\|_{L^\infty_tL^\infty_x}$ and
$\|\nabla v\|_{L^\infty_tL^\infty_x}$. Moreover, $u\in C([0,T];L^2)$, and
\[
\nabla \widetilde{p}
=\nabla(-\Delta)^{-1}\nabla\cdot\big(v\cdot\!\nabla \widetilde{u}\big)
\in L^2(0,T;H^{-1}).
\]
\end{proposition}

\begin{proof}
Using $\mathcal T f:=\nabla(-\Delta)^{-1}\nabla\cdot f$, we
may rewrite \eqref{v_system} as
 $\widetilde{u}_t-\Delta \widetilde{u}= -F(\widetilde{u})$, where
\[
F[\widetilde{u}]\coloneqq  v\cdot\!\nabla \widetilde{u} + \mathcal T\big(v\cdot\!\nabla \widetilde{u}\big).
\]
Using $v\cdot\!\nabla \widetilde{u}=\operatorname{div}(\widetilde{u}\otimes v)-(\operatorname{div}v)\widetilde{u}$
and the boundedness of $\mathcal T$ on $H^{-1}$,
\begin{equation}
\|F[a]\|_{H^{-1}}
\lesssim \big(\|v\|_{L^\infty}+\|\nabla v\|_{L^\infty}\big)\,\|a\|_{2}.
   \label{EQ06}
     \end{equation}
Thus, denoting by $\Phi [a]$ the solution of the heat equation $\p_t b - \Delta b = -F[a]$ with initial data $\widetilde{u}_0$, we have, by a simple energy estimate and \eqref{EQ06} that
\eqnb\label{est_en_00}
\|\Phi(a_1 )-\Phi(a_2 )\|_{X_T}\leq C \| F[a_1 -a_2 ] \|_{L^2 ((0,T);H^{-1})} \leq C L T^{1/2}\|a_1 -a_2 \|_{X_T},
\eqne
where $L\coloneqq \|v\|_{L^\infty}+\|\nabla v\|_{L^\infty}$.
We fix any $T_0\in (0, T]$ such that $CLT_0^{1/2}\leq 1/2$. Then $\widetilde{u} \mapsto \Phi [\widetilde{u}]$ is a contraction on $X_{T_0}$, and an energy estimate similar to \eqref{est_en_00} gives that
\eqnb\label{est_en_01}
\sup_{[0,T_0]} \| \widetilde{u} \|_2^2 + \int_0^{T_0} \| \nabla \widetilde{u} \|^2_2 \leq \| u_0 \|_2^2 + C L \int_0^{T_0} \| \widetilde{u} \|_2^2.
\eqne
We now iterate the above argument over time intervals of length less than or equal $T_0$ to reach $[0,T]$. Iterating \eqref{est_en_01} we obtain a similar estimate valid on $[0,T]$, and an application of the Gronwall inequality gives~\eqref{tildeu_L2est}, as required. 
\end{proof}

We will use Proposition~\ref{prop:lin-v} with initial data of the form
\eqnb\label{u_0cutoff}
\widetilde{u}_0 = u_0 Z_n,
\eqne
where $u_0 \in L^2_{\rm loc} (\R^3)$, $n\geq 10$, and  $Z_n$ denotes any smooth radial non-increasing function such that 
\eqnb\label{choice_Zn}
Z_n(x)=\begin{cases}1 \qquad &|x|\leq n/2,\\
0& |x|\geq n-1,
\end{cases} \quad \text{ and }\quad \|\nabla Z_n\|_{\infty}\leq
\frac{4}{n}.
\eqne
Moreover, we will only consider $v\in L^\infty(0,T;W^{1,\infty}(\mathbb{R}^3))$ satisfying the following conditions for $t\in[0,T]$:
\begin{eqnarray}
v(x,t)=0\quad \text{ for }\quad |x|\geq n,&\label{v1condition}\\
\alpha_t [ v ] + \beta_t [ v ]\leq C_0 \|{u}_0\|^2_{M},&\label{v2condition}\\
\alpha_t [\div v ] + \beta_t [\div  v ]\leq C_0 n^{-2}\|{u}_0\|^2_{M},&\label{v3condition}
\end{eqnarray}
for some constant $C_0\geq 1$, whose value will not change from line to line (it is fixed in~\eqref{choice_C0} below), where we used the notation~\eqref{alphabeta}.

\subsection{Decay estimates on $w= \mathrm{div}  \,\widetilde{u}$}\label{sec_lin_decay}

Here we study the decay of $w=\mathrm{div}  \,\widetilde{u}$. Let $\QQ\in \sc$, and fix $\phi_{\QQ}$ be such that $\phi_{\QQ}=1$ on $\QQ$, $\phi_\QQ=0$ outside of $\QQ^{(3)}$ with
    \begin{equation}\label{phiqrep}
    	\|D^{k}\phi_\QQ\|_{L^{\infty}}\lec_{\sigma,\eta } \frac{C(k )}{|\QQ|^{k /3}}
    \end{equation}
for $k\geq 0$.
We first note that since $w$ solves the heat equation we have
\eqnb\label{w_ptwise}
\begin{split}
|w(x,t)| &= \frac{1}{(4\pi t)^{3/2}}
	\int {\ee^{-\frac{|x-y|^2}{4t}}}\,w(y,0)\,\d y
	\\& \lec t^{-3/2}\ee^{-\dist (\QQ, B(0,n) )^2/4t } \int_{B(0,n)} |w(\cdot ,0)|
	\\&\lec n^2 t^{-3/2}\ee^{-\dist (\QQ, B(0,n) )^2/4t } \| {u_0} \|_{M}
	 \end{split}
\eqne    
for all $x\in B \in \mathscr{C} \setminus  \mathscr{C}_n$, $n\geq 1$,
where we used \eqref{vcond1} and \eqref{u_0cutoff} to get
\[\begin{split}
\int_{B(0,n)} |w(\cdot ,0)| &\lec n^{3/2} \left( \sum_{B' \in \mathscr{C}_n} \int_{B'} | w(\cdot , 0)|^2 \right)^{1/2} \lec n^{3/2} \| w (0) \|_{M} \left( \sum_{B' \in \mathscr{C}_n} | B'|^{2/3} \right)^{1/2}\\
&\lec_\sigma n^{1/2} \| {u_0} \|_{M} |B(0,C_\eta n)|^{1/2} \lec_{\eta,\sigma} n^{2} \| {u_0} \|_{M}.
\end{split}
\]
Apart from \eqref{w_ptwise}, we provide an estimate for $w$ in~$M$. 

\begin{proposition}\label{prop_wbound}
If  $T\leq 1$, then the solution $w$ to \eqref{div} satisfies 
	\eqnb\label{wbound_est}
		\alpha_t [ w ] + \beta_t [ w ]	\lec_{\sigma,\eta } 
		n^{-2} {\|{u_0}\|^2_{M}}.
	\eqne
\end{proposition}
We emphasize that, since the equation for $w$ is independent of $v$ and $\widetilde{u}$, the implicit constant in~\eqref{wbound_est} depends only on $\sigma$,~$\eta$.
\begin{proof}
We multiply \eqref{div} by $w\, \phi_\QQ$ and integrate by parts in space to obtain
\begin{equation}\label{mainw}
	\frac12 \frac{\d}{\d t} \int |w|^2 \phi_\QQ  =\int  \Delta w \, w \, \phi_\QQ  = 
	-\int |\nabla w|^2 \phi_\QQ
   + \frac12 \int w^2 \,\Delta\phi_\QQ
.
\end{equation}
Integration in time and  \eqref{phiqrep} give 
\begin{equation*}
\begin{split}
	\frac12 \int |w(\cdot ,t)|^2 \phi_\QQ
	+ \int_0^t\!\!\int |\nabla w|^2 \phi_\QQ 
	&\leq
	\frac12 \int |w(\cdot ,0)|^2 \phi_\QQ
	+\frac{C}{|\QQ|^{2/3}}\int_0^t \int_{\QQ^{(3)}} |w|^2 .
\end{split}
\end{equation*}
Dividing by $|\QQ|^{2/3}$ and using the fact that each ball $\QQ$ intersects at most $\sigma$ balls, we get
\[
 	\frac{1}{2}\sup_{0<s<t} \|w(s)\|_{M}^2
	+\sup_{\QQ\in\sc}\frac{1}{\QQ^{2/3}}\int_0^t\int_{\QQ}|\nabla w|^2
	\lec_{\sigma,\eta }
	\|w(0)\|_{M}+t \sup_{0<s<t} \|w(s)\|_{M}^2.
\]
Thus, for sufficiently small $t$, depending on $\sigma$ and $\eta $ only, the last term on the right-hand side can be absorbed by the first term on the left-hand side to obtain
\[
 	\alpha_t [ w ] + \beta_t [ w ]
	\lec_{\sigma,\eta} \|w(0)\|_{M}.
\]
 The claim now follows by iterating the above bound to reach the final time $T$ (which can be obtained in at most $O_{\sigma,\eta }(1)$ steps), and by noting that the initial condition in~\eqref{div} gives that $\| w(0)\|_{M}
	\leq \|u_0\|_{M}\|\nabla Z_n\|_{\infty}
	\lec \|u_0\|_{M} / n$. 
\end{proof}

\subsection{The pressure estimate}\label{sec_lin_p}
We now move on to estimating the pressure.
Recalling the definition \eqref{pressure_pn} of $\widetilde{p}$, we shall use a decomposition for a given $\QQ\in \sc$, 
	\begin{equation}\label{localpressapprox}
		\begin{split}
\widetilde{p}(x,t)
		&-
		\widetilde{p}_{\QQ}(t)
		-\widetilde{q}_{\QQ}(t)
	=\widetilde{G}^{\QQ}_{ij}(v_i\widetilde{u}_j)
,
\end{split}
\end{equation}
where
	\begin{equation} \label{EQ08}
		\begin{split}
\widetilde{G}^{\QQ}_{ij}(v_i\widetilde{u}_j)                 
			 &\coloneqq 
     			 -\frac{1}{3}\delta_{ij}f(x)
      			+\text{p.v.}\int_{y\in \QQ^{(7)}}K_{ij}(x-y)(v_i(y,t)\widetilde{u}_j(y,t))\,\d y
      			\\&\indeq
      			+\int_{y\notin \QQ^{(7)}}\left(K_{ij}(x-y)-K_{ij}(x_\QQ-y)\right)(v_i(y,t)\widetilde{u}_j(y,t))\,\d y
			\\&\indeq
			- \int_{y\in \QQ^{(7)}}L_{j}(x-y)\left(\left(\div v(y,t)\right)\widetilde{u}_j(y,t)\right)\,\d y
			\\&\indeq
		 +  \int_{y\notin \QQ^{(7)}}\left(L_{j}(x_\QQ-y)-L_{j}(x-y)\right)\left(\left(\div v(y,t)\right)\widetilde{u}_j(y,t)\right)\,\d y
			\\&=:
			I_1+I_2+I_3+I_4.
			\end{split}
		\end{equation}
We will show that, under the assumption \eqref{ass_extra}, i.e., $|\QQ|^{1/3}\lesssim |x_\QQ|^{1-\epsilon}$ for some $0<\epsilon<1$, and with $n$ sufficiently large, we have
 \begin{equation}\label{pres}
      \begin{split}
          \frac{1}{|\QQ|^{1/3}}\int_0^t\int_{\QQ^{(3)}}|\widetilde{p}-\widetilde{p}_{\QQ}(t)-\widetilde{q}_{\QQ}(t)|^{3/2}
          &\lec_\epsilon \sqrt{C_0}
   t^{1/4}  |\QQ|^{2/3}\| {u}_0 \|_{M}^{3/2}(\alpha_t+\beta_t)^{3/4}
   ,
         \end{split}
  \end{equation}
where we used (and will continue to use in this and the following sections) the short-hand notation
     \eqnb\label{shorthand}
 \alpha_t \coloneqq \alpha_t [ \widetilde{u} ]  ,\qquad \beta_t \coloneqq  \beta_t [ \widetilde{u} ] .
 \eqne
\begin{proof}
To bound $I_1$, we use the Calder\'on-Zygmund inequality to obtain
  \[
  \begin{split}
  	\frac{1}{|\QQ|^{1/3}}\int_0^t\int_{\QQ^{(3)}}|I_1|^{3/2}
    	&\lec
   	\frac{1}{|\QQ|^{1/3}}\int_0^T\int_{\QQ^{(7)}}|v\widetilde{u}|^{3/2}
	\\ &
	\lec  \left(\frac{1}{|\QQ|^{1/3}}\int_0^T\int_{\QQ^{(7)}}|v|^{3}\right)^{1/2}\left(\frac{1}{|\QQ|^{1/3}}\int_0^T\int_{\QQ^{(7)}}|\widetilde{u}|^{3}\right)^{1/2}
.
   \end{split}
          \]
As in~\eqref{ap4}, the  Gagliardo-Nirenberg inequality implies that 
 \begin{equation}\label{gn1}
    	\frac{1}{|\QQ|^{1/3}}\int_0^t\int_{\QQ}|\widetilde{u}|^3
	\lec  t^{1/4}  |\QQ|^{2/3}(\alpha_t+\beta_t)^{3/2}.
  \end{equation}
Using the same idea for $v$ and recalling the assumption \eqref{v2condition} on $v$, we get
  \begin{equation}\label{gn2}
  \frac{1}{|\QQ|^{1/3}}\int_0^t\int_{\QQ}|v|^3
	\lec C_0^{3/2} t^{1/4}  |\QQ|^{2/3} \| {u_0} \|_{M}^{3} .
  \end{equation}
Therefore,
  \[
  	\frac{1}{|\QQ|^{1/3}}\int_0^t\int_{\QQ^{(3)}}|I_1|^{3/2}
	\lec C_0^{3/4}
	t^{1/4}  |\QQ|^{2/3}\| {u_0} \|_{M}^{3/2}(\alpha_t+\beta_t)^{3/4}.
  \]
For $I_2$, recall that for $y\notin \QQ^{(7)}$ and $x\in \QQ^{(3)}$ we have
         \[
           | K_{ij}(x-y)-K_{ij}(x_\QQ-y)|
    \leq 
   \frac{C|\QQ|^{1/3}}{|x-y|^4}.
         \]
Thus,
          \begin{equation}\label{sum}
             \begin{split}
                 |I_2(x)|&\leq C|\QQ|^{1/3}\int_{\mathbb{R}^3\backslash \QQ^{(7)}}\frac{1}{|x-y|^4}|v(y)\widetilde{u}(y)|\d y\\
                 &\lec 
    		|\QQ|^{1/3} \sum_{\QQ'\in \mathscr{C} \setminus P^{(7)}}\int_{\QQ'}\frac{1}{|x_\QQ-x_{\QQ'}|^4}|v(y) \widetilde{u}(y)|\d y\\
                 &\lec
                 \sqrt{C_0}
    		|\QQ|^{1/3} \| {u_0} \|_{M}\sqrt{\alpha_t }\sum_{\QQ'\in \mathscr{C} \setminus P^{(7)}}\frac{|\QQ'|^{2/3}}{|x_{\QQ}-x_{\QQ'}|^{4}}
        ,
   	      \end{split}
         \end{equation}
         where we used \eqref{v2condition} again, and we omitted $t$ in the notation. Repeating the arguments in~\eqref{cubetrans}, we obtain~\eqref{EQ04}, which then gives
         \[
         	|I_2(x)|\leq \sqrt{C_0} \|{u_0} \|_{M}\sqrt{\alpha_t}.
         \]
Therefore,
         \[
         	\frac{1}{|\QQ|^{1/3}}\int_0^t\int_{\QQ^{(3)}}|I_2|^{3/2}
		\lec
		C_0^{3/4}
		t|\QQ|^{2/3}  \|{u_0} \|^{3/2}_{M}\alpha^{3/4}_t.
         \]
 
 For $I_3$, using the fact that $\div v(y,s)=0$ for all $|y|\geq n$ (recall~\eqref{v1condition}), and the Young's inequality, we obtain
\begin{equation}\label{i3}
\begin{split}
\int_{\QQ^{(3)}} | I_3 |^{\frac32} &\leq \int \left| \int 1_{|x-y| \lec_\eta |\QQ|^{1/3}} 1_{y\in Q^{(7)}\cap B(0,n)} \frac{|(\div v(y)) \widetilde{u} (y)|  }{|x-y|^2}   \d y\right|^{\frac32} \d x \\
&\lec \left( \int_{|z|\lec_\eta |\QQ|^{1/3}} |z|^{-2} \right)^{\frac32}  \int_{\QQ^{(7)}} |(\div v) \widetilde{u}|^{\frac32} \\
&\lec |\QQ|^{1/2} \left( \int_{\QQ^{(7)}} |\div v|^3 \cdot  \int_{\QQ^{(7)}} |\widetilde{u}|^3 \right)^{\frac12}.
\end{split}
\end{equation}
Integration in time gives
\eqnb\label{temp010}
\frac{1}{|\QQ|^{1/3}} \int_0^t \int_{\QQ^{(3)}} | I_3 |^{\frac32} \lec |\QQ|^{1/2} \left( \frac{1}{|\QQ|^{1/3}} \int_0^t \int_{\QQ^{(7)}} |\div v|^3 \cdot \frac{1}{|\QQ|^{1/3}} \int_0^t \int_{\QQ^{(7)}} |\widetilde{u}|^3 \right)^{\frac12}.
\eqne
As in~\eqref{ap4}, the  Gagliardo-Nirenberg inequality implies that 
 \[
    	\frac{1}{|\QQ'|^{1/3}}\int_0^t\int_{\QQ'}|\widetilde{u}|^3
	\lec  t^{1/4}  |\QQ'|^{2/3}(\alpha_t+\beta_t)^{3/2}
  \]
for every $\QQ' \in P^{(7)}$. We use the same idea for $\div v$ and the assumption \eqref{v3condition} on $v$ to obtain
  \[
  \frac{1}{|\QQ'|^{1/3}}\int_0^t\int_{\QQ'}|\div v|^3
	\lec C_0^{3/2} t^{1/4}  |\QQ'|^{2/3} n^{-3} \| {u_0} \|_{M}^{3} 
  \]
  for $\QQ' \in P^{(7)}$. Thus, covering $\QQ^{(7)}$ by balls $\QQ' \in P^{(7)}$, \eqref{temp010} gives 
  \begin{equation}\label{i31}
 \frac{1}{|\QQ|^{1/3}} \int_0^t \int_{\QQ^{(3)}} | I_3 |^{\frac32} 
 \lec_{\sigma,\eta} C_0^{3/4} t^{1/4} |\QQ|^{7/6} n^{-3/2}\| \widetilde{u}_0 \|_{M}^{3/2} (\alpha_t+\beta_t)^{3/4}.
  \end{equation}
 Observe that the integral in the first inequality of \eqref{i3} is non-zero only for $\QQ$ such that $\QQ^{(7)}\cap B(0,n)\neq \emptyset$. Let $y\in \QQ^{(7)}\cap B(0,n)$ and $\QQ'\in\mathscr{C}$ such that $y\in\QQ'$. Since $\QQ'\cap\QQ^{(7)}\neq\emptyset$, we have $\QQ'\subset\QQ^{(7)}$. Additionally, since $y\in\QQ'$ and $|y|\leq n$, Proposition~\ref{prop_CR} gives that $|x_{B'} |\lec n$, and so also $|x_{\QQ}|\lec n$. Since $|\QQ|^{1/3}\lesssim|x_\QQ|^{1-\epsilon}$, we therefore obtain $|\QQ|^{1/3}\lesssim n^{1-\epsilon}$. Substituting this into \eqref{i31}, we get
 \begin{equation}\label{i3est}
 	 \frac{1}{|\QQ|^{1/3}} \int_0^t \int_{\QQ^{(3)}} | I_3 |^{\frac32} 
 \lec_{\sigma,\eta} {C_0}{3/4}  t^{1/4} |\QQ|^{2/3} n^{-3\epsilon/2}\| {u_0} \|_{M}^{3/2} (\alpha_t+\beta_t)^{3/4}.
 \end{equation}

Next, we estimate~$I_4$. Recalling \eqref{localpress},
\begin{equation}\label{i4def}
I_4 (x)
= -\frac{1}{\omega_d}\,\int_{(\QQ^{(7)})^c}\underbrace{\left(
     \frac{(x_j-y_j)}{|x-y|^{3}}
    -\frac{(x_{\QQ,j}-y_j)}{|x_\QQ-y|^{3}}
 \right)}_{=: A_j}
 ( \div v(y)) \widetilde{u}(y)\,\d y .
\end{equation}
Note that 
\begin{align*}
A_j
&= \frac{x_j-y_j}{|x-y|^3}-\frac{x_{\QQ,j}-y_j}{|x_\QQ-y|^3}
 = \left( \frac{x_j-x_{\QQ,j}}{b^3}\right) + (x_j-y_j)\left(\frac{1}{|x-y|^3}-\frac{1}{|x_\QQ-y|^3}\right),
\end{align*}
so that, since
\[
|x-x_\QQ| \lec_\eta |\QQ |^{1/3}, \qquad | x-y | \sim_\eta |x_\QQ - y |
,
\]
for $x\in \QQ^{(3)}$, $y \not \in \QQ^{(7)}$, we can use the Mean Value Theorem to get 
\eqnb\label{newdel}
\begin{split}
|A_j | &\lec_\eta \frac{|\QQ|^{1/3}}{|x-y|^3}  + |x-y| \frac{|x_\QQ - x |}{\min \{ |x-y |, |x_Q-y|\}^4} \lec_\eta  \frac{|\QQ|^{1/3}}{|x-y|^3} .
\end{split}
\eqne
Substituting~\eqref{newdel} into \eqref{i4def} and noting $\div v(y)=0$ for $|y|\geq n$, we get
\begin{equation}\label{EQx}
\begin{split}
|I_4 (x)|
&\lec |\QQ|^{1/3} \int_{(\QQ^{(7)})^c}\frac{|\div v(y)|\,|\widetilde{u}(y)|}{|x-y|^{3}}
\,\,\d y
\\ &
\lec_\eta |\QQ|^{1/3}\sum_{\QQ'\in \sc_{n} \setminus P^{(7)} } \frac{1}{|x_\QQ-x_{\QQ'}|^{3}} \int_{\QQ'} |\div v(y)||\widetilde{u}(y) |\,\d y 
\\ &
\lec  |\QQ|^{1/3}\|\div v\|_{M}\|\widetilde{u}\|_{M}\sum_{\QQ'\in \sc_{n}\setminus P^{(7)} } \frac{|\QQ'|^{2/3}}{|x_\QQ-x_{\QQ'}|^{3}} \\
&  \lec  |\QQ|^{1/3}\|\div v\|_{M}\|\widetilde{u}\|_{M}\underbrace{\sum_{\QQ'\in \sc_{n}\setminus P^{(7)} }\sum_{R\in\mathcal{E}_{\QQ'}} \frac{1}{|x_\QQ-x_{R}|^{3}}}_{=: L},
\end{split}
\end{equation}
where we also used \eqref{EQ01} in the second step, as well as the same idea as in  \eqref{cubetrans} in the last step. In order to estimate $L$, we first recall from~\eqref{vcond1} that 
\eqnb\label{vcond11}
|x_{\QQ'} |, r_{\QQ'} \leq C_\eta n
,
\eqne
for some $C_\eta >0$. We consider two cases dependent on the location of $B$ (recall $B\in \mathscr{C}$ is fixed above \eqref{localpressapprox}). \\

\noindent\texttt{Case  $|x_\QQ|\leq 3 C_\eta n$.} Then 
\[
| x_{\QQ}-  x_{R} | \leq | x_{\QQ}-  x_{\QQ'} | \leq | x_{\QQ} | +  |x_{\QQ'} |  \lec_\eta n 
,
\]
for each $R\in \mathcal{E}_{\QQ'}$, $\QQ'\in \mathscr{C}_n\setminus P^{(7)}$, so that $R \in \bigcup_{k=1}^K A_k$ for each such $R$, where $K\lec_\eta n$. As in~\eqref{cubetrans}, we have that $|x_B - x_R| \sim (|B|^{1/3} + k)$ for each $R\in A_k$, and so, recalling Definition~\ref{Cr}(i),
\[
L\lec_\sigma \sum_{k=1}^K \sum_{R\in A_k} (|B|^{1/3} + k )^{-3} \lec \log\left(1+\frac{K}{|\QQ|^{1/3}}\right) \lec \log n.
\]

\noindent\texttt{Case  $|x_\QQ|> 3C_\eta n$.} Then the triangle inequality and \eqref{vcond11} imply that
\[
|x_B - x_R | \geq |x_B - x_{B'} | - r_{B'} \geq | x_B | - |x_{B'}| - r_{B'} \geq \frac{|x_B |}{3}.
\]
Moreover, \eqref{vcond11} implies $B'\subset B(0,2C\eta n)$, and hence the same is true for each $R\in \mathcal{E}_{B'}$. Thus the number of all such $R$'s, as $B' \in \mathscr{C}_n\setminus P^{(7)}$ is bounded by $c_\eta n^3$, for some $c_\eta >0$. Hence,
\[
L \lec_{\sigma,\eta } \frac{n^3}{|x_B|^3} \lec_\eta 1.
\]
Applying these bounds in  \eqref{EQx}, we obtain
\[
|I_4 (x) | \lec_{\sigma, \eta } |B|^{1/3} \|\div v\|_{M}\|\widetilde{u}\|_{M}  \log n  \lec \sqrt{C_0} n^{-\epsilon/2} \|\widetilde{u}_0\|_{M}\alpha_t^{1/2}
\]
for all $x\in B^{(3)}$, where we also used the estimate \eqref{wbound_est} on $\| \div  v\|_{M}$. Integration over $(0,t)\times B^{(3)}$ thus gives 
\begin{equation}\label{i4est}
	\frac{1}{|\QQ|^{1/3}}\int_0^t\int_{\QQ^{(3)}} |I_4|^{3/2}
	\lec_\eta C_0^{3/4} t|\QQ|^{2/3}n^{-3\epsilon/4}\|{u_0}\|^{3/2}_{M}\alpha_t^{3/4},
\end{equation}
as required.
\end{proof}

\subsection{Estimate in $M$}\label{sec_lin_apr}

Here we establish an estimates in $M$ for solutions $\widetilde{u}$ of~\eqref{v_system}.  

\begin{theorem}\label{approxapriori}
Let $u_0\in M$ be divergence-free, $\widetilde{u}_0$ be given by the cutoff \eqref{u_0cutoff} on $B(0,n)$ of $u_0$ and $v\in L^\infty ([0,T]; W^{1,\infty} )$ satisfy assumptions \eqref{v1condition}--\eqref{v3condition} with a constant $C_0\geq 1$. Then, for any $t\in [0,T]$, where
\[
		T
		\coloneqq \frac{1}{C}\min
		\left\lbrace 1,\|u_0\|_{M}^{-4}\right\rbrace
	\]
and $C\geq 1$ is a sufficiently large constant, depending only on $\sigma,\eta$, and  sufficiently large $n$, the solution $\widetilde{u}$ to \eqref{v_system} on $[0,t]$ satisfies
 \begin{equation}\label{apr}
    	\alpha_t + \beta_t 
	\leq 
	C_1 \sqrt{C_0}\|{u_0}\|_{M}
,
    \end{equation}
for some constant $C_1=C_1(\sigma,\eta )>0$.
\end{theorem}
Recall \eqref{shorthand} that in this section we use the shorthand notation  $ \alpha_t \equiv \alpha_t [ \widetilde{u} ]$, $\beta_t \equiv  \beta_t [ \widetilde{u} ]$.

 \begin{proof}[Proof of Theorem~\ref{approxapriori}] 	Fix $\QQ\in \sc$.  Taking the inner product of the first equation in~\eqref{v_system} with $\widetilde{u}\,\phi_\QQ$ and noting that $\nabla \widetilde{p}(x,t)=\nabla(\widetilde{p}(x,t)-\widetilde{p}_{\QQ}(t)-\widetilde{q}_{\QQ}(t))$, we obtain the local energy equality
	\begin{equation}\label{app1}
	\begin{split}
&
		\frac{1}{2}\int|\widetilde{u}(t)|^2\phi_{\QQ}+\int_0^t\int |\nabla \widetilde{u}|^2\phi_\QQ
\\&\indeq
			=
			\frac{1}{2}\int|\widetilde{u}_0|^2\phi_\QQ
			+\frac{1}{2}\int_0^t\int |\widetilde{u}|^2\Delta\phi_\QQ
			+\frac{1}{2}\int_0^t\int|\widetilde{u}|^2(\div v)\phi_\QQ
			\\  &\indeq\indeq
			+\frac{1}{2}\int_0^t\int|\widetilde{u}|^2(v\cdot\nabla\phi_\QQ)
			+\int_0^t\int (\widetilde{p}-\widetilde{p}_{\QQ}-\widetilde{q}_{\QQ})(u\cdot\nabla\phi_\QQ)
			\\  &\indeq\indeq
			+\int_0^t\int (\widetilde{p}-\widetilde{p}_{\QQ}-\widetilde{q}_{\QQ})(\div \widetilde{u})\phi_\QQ
,
	\end{split}
	\end{equation}
where $\phi_\QQ$ is defined above~\eqref{phiq}.	We bound the necessary terms on the right-hand side of~\eqref{app1}. Using \eqref{phiq} and the conditions (i)--(ii) of Definition~\ref{Cr}, we see that
	\[
	\begin{split}
		\frac{1}{2}\int_0^t\int |\widetilde{u}|^2\Delta\phi_\QQ
		&\lec_\sigma
		\frac{t}{|\QQ|^{2/3}}\esssup_{0<s<t}\sup_{\QQ'\subset \QQ^{(3)}}\int_{\QQ'}|\widetilde{u}(s)|^2 \lec_{\eta}  t \alpha_t .
		\end{split}
	\]
Next, we estimate the term with~$\div v$. Applying the Gagliardo-Nirenberg inequality as in~\eqref{gn1}, \eqref{gn2}, and recalling the assumption \eqref{v3condition} on $v$, we obtain
	\[
	\begin{split}
		&\frac{1}{2}\int_0^t\int (\div v)|\widetilde{u}|^2\phi_\QQ
		\\ &\lec
		|\QQ|^{1/3}\left(\sup_{\QQ'\subset \QQ^{(7)}}\frac{1}{|\QQ'|^{1/3}}\int_0^t\int_{\QQ'}|\div v|^3\right)^{1/3}\left(\sup_{\QQ'\subset \QQ^{(7)}}\frac{1}{|\QQ'|^{1/3}}\int_0^t\int_{\QQ'}|\widetilde{u}|^3\right)^{2/3}
		\\ &
		\lec_{\sigma,\eta } \sqrt{C_0}
		t^{1/4}n^{-1}|\QQ|\|{u}_0\|_{M}(\alpha_t+\beta_t)
.
		\end{split}
	\]
Observe that since $v$ is zero outside $B(0,n)$, we have $\QQ'\in\sc_n$ in the above inequality. Thus, Proposition~\ref{prop_CR} gives that $|x_{B'}| \lec n$ and thus also $|x_B|\lec n$. Therefore, recalling the assumption \eqref{ass_extra}, we have  $|\QQ|^{1/3}\lec n^{1-\epsilon}$, giving us
\eqnb\label{div_term_apr}
	\begin{split}
		&\frac{1}{2}\int_0^t\int (\div v)|\widetilde{u}|^2\phi_\QQ
		\lec_{\sigma,\eta } \sqrt{C_0}
		t^{1/4}n^{-\epsilon}|\QQ|^{2/3}\|{u}_0\|_{M}(\alpha_t+\beta_t)
		.
		\end{split}
	\eqne
Next, we estimate the fourth term on the right hand side of~\eqref{app1}. Applying the Gagliardo-Nirenberg inequality as in~\eqref{gn1},\eqref{gn2} and using the assumption \eqref{v2condition} on $v$, we obtain
\[
\begin{split}
&	\frac{1}{2}\int_0^t\int|\widetilde{u}|^2|v\cdot\nabla\phi_\QQ|
	\lec
	\frac{1}{|\QQ|^{1/3}}\int_0^t\int_{\QQ^{(3)}}|\widetilde{u}|^2|v|
	\\ &\indeq
	\lec \left(\sup_{\QQ'\subset \QQ^{(7)}}\frac{1}{|\QQ'|^{1/3}}\int_0^t\int_{\QQ'}| v|^3\right)^{1/3}\left(\sup_{\QQ'\subset \QQ^{(7)}}\frac{1}{|\QQ'|^{1/3}}\int_0^t\int_{\QQ'}|\widetilde{u}|^3\right)^{2/3}
	\\ &\indeq
	\lec \sqrt{C_0} t^{1/4}|\QQ|^{2/3}\|{u}_0\|_{M}(\alpha_t+\beta_t).
\end{split}
\]
For the pressure term in~\eqref{app1} involving $w=\div\widetilde u$, we recall \eqref{w_ptwise} that $w$ admits the Gaussian decay in space for $|x| > n$. We thus consider two cases of the location of $B$, depending whether or not $B^{(3)}$ intersects $B(0,2n)$.\\

\noindent\emph{Case $B^{(3)} \cap B(0,2n) \not = \emptyset$.} In other words $B' \in \mathscr{C}_{2n}$ for some $B' \in P^{(3)}$. \\\

By \eqref{vcond1} this means that $|x_{B'}|,r_{B'} \lec_\eta n$, and thus also $|x_B| \lec_\eta n$. Using the pressure estimate~\eqref{localpressapprox}, the $M$ bound on $w$, and the Gagliardo-Nirenberg inequality as in~\eqref{gn1}  to obtain
\[
\begin{split}
&	\left| \int_0^t\int (\widetilde{p}-\widetilde{p}_{\QQ}-\widetilde{q}_{\QQ})w \phi_\QQ\right|
\\&\indeq
	\lec |\QQ|^{1/3}\left(\frac{1}{|\QQ|^{1/3}}\int_0^t\int_{\QQ^{(3)}}|\widetilde{p}-\widetilde{p}_{\QQ}-\widetilde{q}_{\QQ}|^{3/2}\right)^{2/3}\left(\frac{1}{|\QQ|^{1/3}}\int_0^t\int_{\QQ^{(3)}}|w|^3\right)^{1/3} \\
	&\indeq\lec \sqrt{C_0}t^{1/4} n^{-1} |\QQ|\| {u_0} \|^2_{M}(1+n^{-3\epsilon/2})^{2/3}(\alpha_t+\beta_t)^{1/2}\\
	&\indeq\lec \sqrt{C_0} t^{1/4} n^{-\epsilon } \| {u_0} \|^2_{M}(\alpha_t+\beta_t)^{1/2},
\end{split}
\]
where, in the last line, we also used the assumption \eqref{ass_extra} that $|B| \lec |x_B |^{1-\epsilon}$.\\

\noindent\emph{Case $B^{(3)} \cap B(0,2n) = \emptyset$.} We have $|x_B| \geq 2n$ and, recalling \eqref{EQ01cons},
\[
\dist (B^{(3)},B(0,n)) \gec_\eta  \dist (x_B , B(0,n)) \geq |x_B | - n \geq |x_B |/2
,
\]
we get
 $\dist (B^{(3)}, B(0,n))\geq n$, and so the pointwise estimate \eqref{w_ptwise} on $w$ gives
 \[
 \begin{split}
	&\left| \int_0^t\int (\widetilde{p}-\widetilde{p}_{\QQ}-\widetilde{q}_{\QQ})w \phi_\QQ \right|
	\lec  n^2 t^{-3/2}\ee^{-|x_B|^2/C_\eta t} \| {u_0} \|_{M} \int_0^t \int_{B^{(3)}} |\widetilde{p}-\widetilde{p}_{\QQ}-\widetilde{q}_{\QQ}|\\
	&\indeq\lec n^2t^{-7/6}\ee^{-|x_B|^2/C_\eta t} \| {u_0} \|_{M} |\QQ |^{5/9}\left( \frac{1}{|B|^{1/3}} \int_0^t \int_{B^{(3)}} |\widetilde{p}-\widetilde{p}_{\QQ}-\widetilde{q}_{\QQ}|^{3/2} \right)^{\frac23}
    \\&\indeq
	\lec\sqrt{C_0}  n^2t^{-7/6}\ee^{-|x_B|^2/C_\eta t} \| {u_0} \|_{M} |\QQ |^{5/9}\left( t^{1/4}  |\QQ|^{2/3}\| {u_0} \|_{M}^{3/2}(\alpha_t+\beta_t)^{3/4} \right)^{\frac23}
  ,
	\end{split}
	\]
from where
 \[
 \begin{split}
	&\left| \int_0^t\int (\widetilde{p}-\widetilde{p}_{\QQ}-\widetilde{q}_{\QQ})w \phi_\QQ \right|
	\lec \sqrt{C_0} n^2 t^{-1}\ee^{-|x_B|^2/C_\eta t} |\QQ | \| {u_0} \|_{M}^2(\alpha_t+\beta_t)^{1/2 } \\
	&\indeq\lec_\eta \sqrt{C_0} n^{-\epsilon } |B|^{2/3} |x_B|^{3 } t^{-1}  \ee^{-|x_B|^2/C_\eta t}   \| {u_0} \|_{M}^2(\alpha_t+\beta_t)^{1/2 } \\
	&\indeq\lec_\eta \sqrt{C_0} n^{-\epsilon }|B|^{2/3}   t^{1/2}   \| {u_0} \|_{M}^2(\alpha_t+\beta_t)^{1/2 }
 .
\end{split}
\]
Lastly, using \eqref{localpressapprox}, the Gagliardo-Nirenberg inequality as in~\eqref{gn1} and noting $n\geq 1$, we obtain
\[
\begin{split}
&
	\int_0^t\int |\widetilde{p}-\widetilde{p}_{\QQ}-\widetilde{q}_{\QQ}||\widetilde{u}||\nabla \phi_\QQ|
  \\&\indeq
	\lec
	\left(\frac{1}{|\QQ|^{1/3}}\int_0^t\int_{\QQ^{(3)}}|\widetilde{p}-\widetilde{p}_{\QQ}-\widetilde{q}_{\QQ}|^{3/2}\right)^{2/3}\left(\frac{1}{|\QQ'|^{1/3}}\int_0^t\int_{\QQ'}|\widetilde{u}|^3\right)^{1/3}
	\\ &\indeq
	\lec\sqrt{C_0} t^{1/4}  |\QQ|^{2/3} \| {u_0} \|_{M}(\alpha_t+\beta_t).   
\end{split}
\]
Combining all the estimates, we obtain
   	 \[
  	  \begin{split}
    		&
\frac12
              \int_{\QQ}|\widetilde{u}(t)|^2+\int_0^t\int_{\QQ} |\nabla \widetilde{u}|^2
		\\ &
		\indeq\lec  \frac12 \int |\widetilde{u}_0|^2 \phi_\QQ 
		+t\alpha_t
		+\sqrt{C_0}t^{1/4}|\QQ|^{2/3}\|{u_0}\|_{M}(\alpha_t+\beta_t)
		\\ &\indeq\indeq
		+ \sqrt{C_0}t^{1/4} n^{-1} |\QQ|^{4/9}\exp\left(-n^2/4Ct\right)\| {u_0} \|^2_{M}(\alpha_t+\beta_t)^{1/2}.	
		\end{split}
   	 \]
 Dividing by $|\QQ|^{2/3}$ and recalling that $|\QQ|\geq 1$, we see that
  	  \begin{equation}\label{teem}
  	  \begin{split}
             \frac{1}{|\QQ|^{2/3}} \int_{\QQ}|\widetilde{u}(t)|^2
             &+\frac{1}{|\QQ|^{2/3}}\int_0^t\int_{\QQ} |\nabla \widetilde{u}|^2
				\lec \alpha_0
		+t\alpha_t
		+\sqrt{C_0} t^{1/4}\|{u_0}\|_{M}(\alpha_t+\beta_t)
				.
		\end{split}
   	 \end{equation}
Thus, since $\alpha_t + \beta_t $ is continuous in $t$, we can apply a  barrier argument to obtain that 
\[
	\alpha_t+\beta_t\lec_{\sigma,\eta } \sqrt{C_0} \|{u_0}  \|_{M}^2
\]
for $t\in [0,T]$, where $T\coloneqq \min \{ 1, \|{u_0}  \|_{M}^{-4} \}/C$, and $C=C(\sigma,\eta )>0$ is the implicit constant from~\eqref{teem}. 
 \end{proof}

\subsection{A regularization lemma}\label{sec_reg_lemma}
For $\gamma \in (0,1]$ we denote by $\rho_\gamma(x) \coloneqq \gamma^{-3}\rho (x/\gamma )$ a standard mollifier, where $\rho \in C_c^\infty (B(0,1);\R^+)$ is such that $\int \rho=1$.
 
\begin{lemma}\label{lem_reg} Suppose that 
there exists a constant $C_2 = C_2 (\sigma,\eta )\geq 1$ with the following property: if  $T$ is given by Theorem~\ref{approxapriori}, $t\in (0,T]$, $n\geq 3$ is sufficiently large (as in Theorem~\ref{approxapriori}), $v$ satisfies conditions \eqref{v1condition}--\eqref{v3condition} with a constant $C_0\geq 1$, and $\widetilde{u}$ is a solution to the linear system \eqref{v_system} on $[0,t]$ with initial data $\widetilde{u}_0$ (recall~\eqref{u_0cutoff}),
then 
\[
\widetilde{v} \coloneqq (\widetilde{u}Z_n)\ast \rho_\gamma
\]
satisfies 
 \begin{eqnarray}
\widetilde{v}(x,s)=0\quad \text{ for }\quad |x|\geq n, s\in [0,t],&\label{vtilde1}\\
\alpha_t [ \widetilde{v} ] + \beta_t [ \widetilde{v} ]\leq C_2 C_1 \sqrt{C_0} \|{u_0}\|^2_{M},&\label{vtilde2}\\
\alpha_t [ \div \widetilde{v} ] + \beta_t [ \div \widetilde{v} ]\leq C_2 n^{-2}\|{u_0}\|^2_{M}.&\label{vtilde3}
\end{eqnarray}
\end{lemma} 

\begin{proof}[Proof of Lemma~\ref{lem_reg}]
Note that \eqref{vtilde1} is trivial since $Z_n=0$ on $B (0,n-1)^c$. Given $B\in \sc$, let $\phi_B \in C_c^\infty (B^{(3)};[0,1])$ be such that $\phi_B=1$ on $B$ and let $\psi_B \in C_c^\infty (B^{(9)};[0,1])$ be such that $\psi_B=1$ on~$B^{(6)}$. We also require that $\phi_B$, $\psi_B$ satisfy the $L^\infty$ estimate \eqref{phiq} on derivatives. \\
 
 Note that, by the choice of the cutoffs $\phi_B$, $\psi_B$, we have  $\widetilde{v} \phi_B =  \phi_B (\widetilde{u}Z_{n}\psi_B)\ast \rho_\gamma$, and so
	\[
	\begin{split}
		\| \widetilde{v} \phi_\QQ \|_2 &	\leq \|\widetilde{u} Z_n \psi_\QQ\|_2 
		\end{split}
.
	\]
	Dividing by $|\QQ|^{2/3}$, using that $\widetilde{u}$ satisfies \eqref{apr}, and the fact that balls in $\QQ^{(6)}$ are of comparable size, we obtain
	\[\sup_{0<t<T} \|\widetilde{v}(t)\|_{M}^2
		\lec_{\sigma,\eta } \|\rho_\gamma\|_{L^1}^2 \sup_{0<t<T} \|\widetilde{u}(t)\|_{M}^2
		\leq C_1 \sqrt{C_0}  \|{u}_0\|_{M}^2,
	\]
	where we used \eqref{apr} and that $\|\rho_\gamma\|_{L^1}=1$. Next, we observe that 
	\[
		\nabla \widetilde{v} =\left(\nabla \widetilde{u} Z_n+\nabla Z_n \otimes \widetilde{u}\right)*\rho_\gamma,
	\]
	and so, since $T\leq 1$, we can use \eqref{apr} again to get  
	\[
		\sup_{Q\in\sc}\frac{1}{|Q|^{2/3}}\int_0^t\int_Q|\nabla \widetilde{v}|^2
		\lec   \|{u_0}\|_{M}^2
		,
	\]
	which concludes the proof of~\eqref{vtilde2}. Finally, Proposition~\ref{prop_wbound} lets us use a similar argument to obtain~\eqref{vtilde3}.
 \end{proof}
 
\subsection{The conclusion of the proof of Theorem~\ref{thm_main} in the case~\eqref{ass_extra}}\label{sec_pf_main}

Here we prove Theorem~\ref{thm_main} under the assumption~\eqref{ass_extra}. We note that this assumption was already used in obtaining the pressure estimate \eqref{pres}. Moreover, thanks to \eqref{ass_extra}, we saw in \eqref{i3est} and \eqref{i4est} that the estimates on the terms $I_3$ and $I_4$ in the pressure estimate involving kernel $L$ in the pressure decomposition \eqref{EQ08} vanish as $n\to \infty$. The same holds for the estimate~\eqref{div_term_apr} on the term  involving $\div\,v$ in the a~priori estimate \eqref{apr}. This means that, except for the a~priori estimate \eqref{apr} remaining uniform with respect to $n$, all terms involving $\div\,v$, will vanish as $n\to \infty$. In other words,  roughly speaking, the error we make by using our non-divergence-free system \eqref{v_system} disappears in the limit $n\to \infty$. This indicates that, by using an appropriate retarded mollification scheme, as described in Step 1 below, the construction procedure should converge to a divergence-free solution (see Step 2 below).\\ 

Let $u_0\in M$, and suppose that
$N\in \N$ is sufficiently large so that it satisfies the hypothesis of Theorem~\ref{approxapriori}. Let $T>0$ be given by Theorem~\ref{approxapriori}.
Also, fix $K\geq 1$ and $\gamma \in (0,1)$. For $n\geq N$ we truncate $u_0$ as in~\eqref{u_0cutoff}, i.e., $\widetilde{u}_0 \coloneqq u_0 Z_n$, where $Z_n$ the cutoff function of $B(n/2)$ inside $B(n-1)$; recall~\eqref{choice_Zn}. 

We set
\eqnb\label{choice_C0}
C_0 \coloneqq  C_1^2 C_2^2.
\eqne

\noindent\texttt{Step~1.} For each sufficiently large $n$, we construct a solution $u_n \in X_T$ to \eqref{v_system} with initial condition $\widetilde{u}_0$ and  $v_n \in X_T \cap L^\infty ([0,T];W^{1,\infty })$ such that $v_n (\cdot ,t) =0$ on $B(0,n)^c$, $v_n (\cdot , t) = (u_n (\cdot ,t-T/K) Z_n )\ast \rho_\gamma$ for $t\in [T/K,T]$, and 
\eqnb\label{un_vn}
\begin{split}
\alpha_t [ {u}_n ] + \beta_t [ {u}_n ]&\leq C_1 \sqrt{C_0} \|{u}_0\|^2_{M},\\
\alpha_t [ {v}_n ] + \beta_t [ {v}_n ]&\leq C_0 \|{u}_0\|^2_{M},\\
\alpha_t [\div {v}_n ] + \beta_t [\div  {v}_n ]&\leq C_0 n^{-2}\|{u}_0\|^2_{M}
,
\end{split}
\eqne
for $t\in [0,T]$. Recall from ~\eqref{XT_def} that $X_T= C ([0,T];L^2)\cap L^2(0,T;H^1)$. \\

Let $\widetilde{v}_n^{(0)}\coloneqq 0$.
For each $k=1,\ldots, K$, we use Proposition~\ref{prop:lin-v} to obtain a solution $u_n^{(k)}$ to~\eqref{v_system} on the time interval $[0,kT/K]$ with initial condition $\widetilde{u}_0$ and $v\coloneqq v_{n}^{(k-1)}(\cdot , t-T/K)$, and we define
\[
{v}_n^{(k)} (\cdot ,t) \coloneqq (u_n^{(k)}(t)Z_n )\ast \rho_\gamma \qquad \text{ for } t\in \left( 0 , \frac{kT}K \right]
\]
and ${v}_n^{(k)} (\cdot ,t) \coloneqq 0$ for $t\in (-T/K,0]$. 

Observe that, by induction, \eqnb\label{ind_v}
\begin{split}
{v}_n^{(k)}(x,t)&=0\quad \text{ for }\quad |x|\geq n,\\
\alpha_t [ {v}_n^{(k)} ] + \beta_t [ {v}_n^{(k)} ]&\leq C_0 \|{u}_0\|^2_{M},\\
\alpha_t [\div {v}_n^{(k)} ] + \beta_t [\div  {v}_n^{(k)} ]&\leq C_0 n^{-2}\|{u}_0\|^2_{M}
,
\end{split}
\eqne
for all $t\in [0,kT/K]$, $k=0,\ldots , K$. Indeed, the base case $k=0$ is trivial. If \eqref{ind_v} holds for $k\geq 0$, then the a~priori estimate \eqref{apr} for $u_n^{(k+1)}$ shows that
\[
\alpha_t [ {u}_n^{(k+1)} ] + \beta_t [ {u}_n^{(k+1)} ]\leq C_1 \sqrt{C_0} \|{u}_0\|^2_{M}
,
\]
for $t\in [0,(k+1)T/K]$, and, consequently, the regularization Lemma~\ref{lem_reg} gives \eqref{ind_v} for~$k+1$. Note that the constants $C_0$ on the right-hand sides of the inequalities in~\eqref{ind_v} are obtained by using our choice of $C_0$, i.e., by noting that 
\[
C_2C_1 \sqrt{C_0} =C_0 \qquad \text{ and } \qquad C_2\leq C_0.
\] 

Having verified \eqref{ind_v}, for all $k=0,\ldots , K$, we see that $u_n^{(k)}$ is well-defined for each $k=1,\ldots , K$. In particular,
\[
u_n \coloneqq u_n^{(K)}, \quad v_n (\cdot, t) \coloneqq v_n^{(K-1)}(\cdot , t- T/K) 
\]
satisfy the claim of this step. \\

\noindent\texttt{Step~2.} We take the limit $n\to \infty$ to obtain $u,v\in X_T$ and $p\in L^{3/2} ((0,T); L^{3/2}_{\loc})$ such that $u$ is weakly divergence free,  \eqref{v_system} holds in a weak sense, 
\eqnb\label{uv_apr}
	\begin{split}
\alpha_t [ {u} ] +  \alpha_t [ {v} ]   + \beta_t [ {u} ]+ \beta_t [ {v} ]&\lec  \|{u}_0\|^2_{M}
,
\end{split}
	\eqne
	for almost every $t\in (0,T)$, for each $B\in \sc $,  $u(t)$ is weakly continuous in time with values in $L^2_{\loc}$, and
	\eqnb\label{v_vs_u}
	v(\cdot , t) = u(\cdot , t-T/K) \ast \rho_\gamma\quad \text{ for almost every }t\in (T/K,T).
	\eqne 
Moreover, $p$ satisfies
	\eqnb\label{p_decc}
	\begin{split}
		p(x,t)-p_\QQ(t)
			 &= -(v\cdot u )(x,t) 
      			+\text{p.v.}\int_{ \QQ^{(7)}}K_{ij}(x-y)(v_i \,u_j)(y,t)\d y 
      			\\&\indeq
      			\hspace{1cm}+\int_{( \QQ^{(7)})^c }\left(K_{ij}(x-y)-K_{ij}(x_\QQ-y)\right)(v_i\,u_j)(y,t) \d y
			\end{split}
	\eqne
for all $x\in B^{(3)}$ and almost every $t\in (0,T)$, where  $p_B (t)$ is a function of time only, and  the local energy inequality holds, i.e., 
			\begin{align}
  			\begin{split}
&		\int |u (t) |^2 \phi (t) +	2\int_0^t \int |\nabla u|^2\phi
        \\&\indeq
			\leq \int |u_0 |^2 \phi (0) +
			\int_0^t \int |u|^2(\partial_t\phi
			+ \Delta\phi)+\int_0^t \int(|u|^2+2p)(v\cdot\nabla\phi)
,
  \end{split}
\label{vu_LEI}
\end{align}
for all non-negative $\phi\in C_0^\infty(\R^3 \times [0,T) )$, and almost every $t\in (0,T)$.

 We emphasize that $u,v$ are dependent on $K,\gamma$; recall the comments in the beginning of this section.

Since $\sc$ is countable, we may denote each ball in $\sc$ by $B_k$ for a unique positive integer~$k$. From Step~1 we see that 
	\[
		\sup_{0<t<T}\int_{B_k}|u_n (t)|^2+\int_0^T\int_{B_k}|\nabla u_n |^2 \leq C(k,u_0).
	\]
	Consequently, we obtain
	\[
		\int_0^T\int_{B_k}|u_n |^{10/3}\leq C(k,u_0)
		\quad\text{and so also}\quad
		\int_0^T\int_{B_k}|v_n |^{10/3}\leq C(k,u_0)
.
	\]
Using the equation  \eqref{v_system} for $u_n$, we thus have that 
	\[
		\int_0^T\int_{B_k}\partial _tu_n\cdot g \leq C(k,u_0)\left(\int_0^T\int_{B_k}|\nabla g|^3 \right)^{1/3}
,
	\]
for any $g\in C_0^\infty(B_k)$ (as in \cite[p. 154]{kikuchi}). This gives that  $\|\partial_t u_{n}\|_{X_k}\leq C(k,u_0)$, where $X_k$ is the dual space of $L^3(0,T;\mathring{W}^{1,3}(B_k))$.
	By a standard diagonal argument, we can thus extract a subsequence $(n_m)$ such that  
	\eqnb\label{conv_uv}
	\begin{split}
		u_{n}
				\overset{\ast}{\rightharpoonup} u\qquad \text{ and } \qquad v_{n}
				\overset{\ast}{\rightharpoonup} v \qquad &\text{in } L^{\infty}(0,T;L^2_{\loc}),
		\\
		u_n
		\rightharpoonup u \quad\text{ and } \qquad v_{n}
				{\rightharpoonup} v \qquad & \text{in } L^{2}(0,T;H^1_{\loc}),
		\\
		u_n
		\rightarrow u \quad\text{ and } \qquad v_{n}
				{\rightarrow} v \qquad & \text{in } L^{3}(0,T;L^3_{\loc}),
	\end{split}
	\eqne
as $k\to\infty$. Note that, by taking the limit in the weak formulation of $v_n (\cdot , t) = (u_n (\cdot ,t-T/K) Z_n )\ast \rho_\gamma$ for $t\in (T/K,T)$ we obtain~\eqref{v_vs_u}. Furthermore, $\div v=0$, due to the last estimate in~\eqref{un_vn}. Moreover, the first two inequalities in~\eqref{un_vn} imply~\eqref{uv_apr}.  In Step~2a below we show that
\eqnb\label{G_conv}
\widetilde{G}_{ij}^\QQ((v_n)_i (u_n)_j )\to G_{ij}^\QQ (v_i u_j )\quad \text{ in }\quad L^{3/2}(0,T;L^{3/2}(\QQ))
,
\eqne
for every $B\in \sc$. Given \eqref{G_conv} we define $p\coloneqq G_{ij}^{\QQ_0} (v_i u_j )$, where $B_0 \in \sc$ is any ball such that $0\in B_0$. To see that $p \in L^{3/2} ((0,T); L^{3/2}_{\loc })$, we first note that \eqref{G_conv} gives that $p\in L^{3/2}(0,T;L^{3/2}(\QQ_0))$. Moreover,
\[
p = G_{ij}^{\QQ'} (v_i u_j ) + c_{B',B_0}(t)
,
\]
for any $B'\in \sc$ such that $B'\cap B_0 \ne \empty$, where 
\[
c_{B',B_0}(t) \coloneqq \int_{(B_0^{(7)} )^c} (K_{ij}(x_{B'} -y ) - K_{ij}(x_{B_0} -y )) (v_iu_j)(y,t) \d y \in L^{3/2} ((0,T)),
\]
by an argument similar to \eqref{36}--\eqref{EQ04aa} above. Thus $p\in L^{3/2}(0,T;L^{3/2}(\QQ'))$, and, by induction, $p\in L^{3/2}(0,T;L^{3/2}(\QQ))$
for any $B\in \sc$. Thus $p\in L^{3/2}(0,T;L^{3/2}_{\loc})$ and $p$ satisfies the decomposition \eqref{p_decc}, as required.\\

\noindent\texttt{Step~2a.} We show~\eqref{G_conv}.\\

Recalling the definitions \eqref{localpress_repeat} and \eqref{localpressapprox} of ${G}_{ij}^\QQ$ and $\widetilde{G}_{ij}^\QQ$, respectively, we need to prove that  
	\begin{equation}\label{pp}
		\begin{split}
			-\frac{1}{3} (v_n \cdot u_n )(x,t)
			&+\underbrace{\text{p.v.}\int_{ \QQ^{(7)}}K(x-y)(v_n u_n )(y,t)\,\d y}_{=:I_1^{(n)}}
      			\\&\hspace{3cm}
      			+\underbrace{\int_{ (\QQ^{(7)})^c}\left(K(x-y)-K(x_\QQ-y)\right)(v_n u_n) (y,t)\,\d y}_{=:I_2^{(n)}}
			\\&
			- \underbrace{\int_{y\in \QQ^{(7)}}L(x-y)\left(\left(\div v_n (y,t)\right)u_n (y,t)\right)\,\d y}_{=:I_3^{(n)}}
			\\&
		\indeq +  \underbrace{\int_{y\notin \QQ^{(7)}}\left(L(x_\QQ-y)-L(x-y)\right)\left(\left(\div v_n (y,t)\right)u_n (y,t)\right)\,\d y}_{=:I_4^{(n)}}
			\\& 
			\hspace{-3cm}\to -\frac{1}{3}(v\cdot  u)(x,t)
			+\underbrace{\text{p.v.}\int_{y\in \QQ^{(7)}}K(x-y)(vu )(y,t)\,\d y}_{=:I_1}
\\&\indeq
      			    			+\underbrace{\int_{y\notin \QQ^{(7)}}\left(K(x-y)-K(x_\QQ-y)\right)(v u )(y,t)\,\d y}_{=:I_2}
		\end{split}
	\end{equation}
	in $L^{3/2}((0,T);L^{3/2} (B))$ as $n\to \infty$, for each $B\in \sc$, where, for brevity, we omitted all indices $i,j$ as subscripts.

	Note that $\int_0^T \int_B \left( |I_3^{(n)}|^{3/2} + |I_4^{(n)}|^{3/2} \right) \to 0$, due to \eqref{i3est} and~\eqref{i4est}. Moreover, the $L^3$-convergence in~\eqref{conv_uv} implies that $v_n\cdot u_n \to v\cdot u $ in  $L^{3/2}((0,T);L^{3/2}_{\loc})$, and an application of the Calder\'on-Zygmund inequality gives also  that
\[
\int_0^T \int_B |I_1^{(n)}- I_1|^{3/2} \to 0.
\]	
As for the convergence $I_4^{(n)}\to I_4$, let $\varepsilon >0$. For $N\geq 1$, we have, as in~\eqref{sum},
\[
\begin{split}
&\int_{B(x_B, N)^c\cup (B^{(7)})^c  } (K(x-y)-K(x_B -y))(v_n u_n )(y) \d y
\\&\indeq
\lec \| u_0 \|_{M}  \sum_{\substack{B' \in \sc \\ B'\cap B(x_B,N)^c \ne \emptyset}}  \frac{|B'|^{2/3}|B|^{1/3}}{|x_B- x_{B'}|^4}
\lec_{B,u_0,\sigma }  \sum_{n \geq N- |B|^{1/3} } \frac{1}{(|B|^{1/3} +n )^2} \leq \varepsilon 
\end{split}
\]
if $N$ is sufficiently large, where, in the first inequality, we used the estimates \eqref{un_vn} of $u_n,v_n$ in $M$, and, in the second inequality, we used the same estimate as in~\eqref{cubetrans}. 
Similarly we can use estimates \eqref{uv_apr} to show that the contribution to $I_4(x)$ from $y\not \in B(x_B, N )$ is negligible for such~$N$. Thus,

	\begin{equation*}
		\begin{split}
			\int_0^T\int_{\QQ}|I_4^{(n)} - I_4 |^{3/2}	&
			\lec |\QQ|\int_0^T\left(\int_{B(x_B,N )}|(v_n u_n )(y,t) - (vu)(y,t)|\,\d y\right)^{3/2}\,\d t
			+ O_{B,u_0,\sigma }(\varepsilon ) \\ &
			\lec_{B,u_0,\sigma } N^{3/2} \int_0^T \int_{B(x_B,N )} | v_n u_n - vu |^{3/2} + O(\varepsilon ) \lec \varepsilon 
		\end{split}
	\end{equation*}
	if $n $ is sufficiently large, as required. \\
	
	\noindent\texttt{Step~3} We take the limit $K\to \infty$ and $\gamma \to 0$ to obtain the claim of Theorem~\ref{thm_main}.\\
	
This step follows by using similar weak compactness techniques as in Step~2. These are standard, and both $u,v$ are now divergence free, and thus we omit the details.

\subsection{Proof of Theorem~\ref{thm_main} in the case~\eqref{ass2}}
\label{sec_ass2}
Here we prove Theorem~\ref{thm_main} under the assumption \eqref{ass2}, i.e., that there exists a constant $\kappa >0$ such that
\[
 |B|^{1/3} \geq \kappa |x_B|
 \]
for all $B\in\sc $ and sufficiently large~$|x_B|$. Let us consider $n\in \N$ sufficiently large so that \eqref{prop_CR} holds for~$\sc_{n}$. We first show the following.
\begin{lemma}\label{lem_finmany}
There exists $L=L(\eta,\kappa )\geq 2$ such that $B \subset  B(L n )$  and $|x_B| \geq  L^{-1} n$ for every $B\in \sc $ intersecting $B(2n) \setminus B(n )$ and every sufficiently large~$n$.

In particular:
\eqnb\label{pigeon}
\text{ there are at most }\frac{4\pi}3 \sigma L^6 \kappa^{-3} \text{ balls }B\in \sc \text{ intersecting } B(2n ) \setminus B(n  ).
\eqne 
\end{lemma}
\begin{proof}
Assume that $B\in \sc $ intersects $B(2n ) \setminus B(n )$. Then Proposition~\ref{prop_CR} shows that  $B\subset B( L n)$ for some $L=L(\eta )\geq 1$. In particular $|x_B| \geq L^{-1} n$, as otherwise $B\subset B( n )$, which is a contradiction. Thus, since every point of $\R^3$ is covered by at most $\sigma$ balls from $\sc$, and $|B| \geq \kappa^3 |x_B|^3 \geq \kappa^3 L^{-3} n^3$ the pigeonhole principle gives~\eqref{pigeon}.
\end{proof}

Thanks to the lemma, we can obtain divergence-free $L^2$ approximations of a given $u_0\in M$.

\begin{lemma}\label{lem_div_free_cutoff}
Given $u_0 \in M$ and sufficiently large $n$, there exists $g\in L^2$ such that $\div\,g=0$, $g=u_0$ on $B(0,n)$ and
\eqnb\label{cutoff_est}
\| g \|_{M} \leq C \| u_0 \|_M,
\eqne
where $C=C(\sigma,\eta, \kappa )\geq 1$ is a constant. 
\end{lemma}

\begin{proof}[Proof of Lemma~\ref{lem_div_free_cutoff}]
Let $Z_n\in C_c^\infty (B(2n);[0,1])$ be such that $Z_n =1$ on $B(n)$ and $\| \nabla Z_n \|_\infty \lec n^{-1}$. Let $\Psi $ be a linear map such that, for every $f \in  L^2 (B(2n)\setminus B(n))$ with $\int_{B(2n)\setminus B(n) }f =0$, we have $\Psi f\in H^1_0 (B(2n)\setminus B(n) )$, $\Psi f$ is divergence-free, and 
\eqnb\label{bogo}
\| \Psi f \|_{H^1_0 (B(2n)\setminus B(n))} \lec n \| f \|_{L^2 (B(2n)\setminus B(n))}.
\eqne
The map $\Psi$ exists by a generalization of the Bogovski\u{\i} theorem~\cite{B}; see Lemma~4.1 in \cite{BK1}, for example, which also gives the constant $O(n)$ in~\eqref{bogo}. Letting 
\[
g\coloneqq u_0 \cdot \nabla Z_n - \Psi (u_0\cdot \nabla Z_n )
,
\]
we obtain all claims of the lemma, except for~\eqref{cutoff_est}. Indeed, note that $\int_{B(2n)\setminus B(n)} u_0 \cdot \nabla Z_n = -\int_{\p B(2n) } u_0 \cdot \mathrm{n} -\int_{\p B(n)} u_0 \cdot \mathrm{n} =0$, so that $\Psi (u_0\cdot \nabla Z_n )  $ is well defined, and $g$ is divergence-free across $\p(B(2n)\setminus B(n))$. 

In order to see \eqref{cutoff_est}, we trivially have $|B|^{-2/3} \int_B |g|^2 = |B|^{-2/3} \int_B |u_0|^2 \leq \| u_0 \|_M^2$ for every $B$ disjoint with $B(2n)\setminus B(n)$. If $B\cap (B(2n)\setminus B(n) ) \ne \emptyset$, then
\[\begin{split}
\int_B |g|^2 &\leq \int_B |u_0 \cdot \nabla Z_n|^2 + \int_{B\cap (B(2n)\setminus B(n))}  |\Psi (u_0 \cdot \nabla Z_n)|^2\\
&\lec |B|^{2/3}n^{-2} \| u_0 \|_M^2 + n^2 \int_{B(2n)\setminus B(n)}  |u_0 \cdot \nabla Z_n|^2\\
&\lec |B|^{2/3}n^{-2} \| u_0 \|_M^2 + \sum_{\substack{B' \in \sc \\ B'\cap (B(2n)\setminus B(n)) \ne \emptyset}} \int_{B' }  |u_0|^2 \\
&\lec_{\sigma,\eta ,\kappa } |B|^{2/3} \| u_0 \|_M^2,
\end{split}
\]
 as required, where, in the last inequality we used \eqref{pigeon}, which also implies all $B'\in\sc$ intersecting $B(2n)\setminus B(n)$ have comparable volumes.
\end{proof}
Theorem~\ref{thm_main} now follows by considering the sequence $(g_n)$, where $g_n$ is obtained by Lemma~\ref{lem_div_free_cutoff} for each sufficiently large~$n$. Indeed, since $g_n\in L^2 (\R^3)$, we can consider the sequence $u_n$ of global-in-time Leray weak solutions (see Theorem~6.37 in \cite{OP}, for example) with initial data $g_n$, obtain the a~priori estimate  \eqref{main_apr} on time interval $[0,T]$. A standard compactness argument (see e.g.~Section~\ref{sec_pf_main}, where it was used repeatedly), gives the desired local-in-time weak solution.
     
\section*{Acknowledgments}
IK was supported in part by the NSF grant DMS-2205493,
while WO was supported the NSF grant no.~DMS-2511556 and by the Simons grant SFI-MPS-TSM-00014233.


\begin{thebibliography}{[66]}
   \bibitem{A}
   Abe, K., 2015. The Navier-Stokes Equations in a Space of Bounded Functions. Communications in Mathematical Physics, 338(2).
   
   \bibitem{A2}
   Abe, K., 2016. On estimates for the Stokes flow in a space of bounded functions. Journal of Differential Equations, 261(3), pp.~1756--1795. 
   
   \bibitem{AG}
   Abe, K. and Giga, Y., 2013. Analyticity of the Stokes semigroup in spaces of bounded functions.
   
   \bibitem{AG2}
   Abe, K. and Giga, Y., 2014. The $L^\infty$-Stokes semigroup in exterior domains. Journal of Evolution Equations, 14, pp.~1--28.
   
   \bibitem{BC}
   Bradshaw, Z., Chernobai, M. and Tsai, T.P., 2025. Global Navier-Stokes flows in intermediate spaces. Journal of Differential Equations, 429, pp.~50--87.
   
   \bibitem{BK1}
   Bradshaw, Z. and Kukavica, I., 2020. Existence of suitable weak solutions to the Navier–Stokes equations for intermittent data. Journal of Mathematical Fluid Mechanics, 22(1), p.3.
   
\bibitem{BK2}
   Bradshaw, Z. and Kukavica, I., 2025.
The structure of weak solutions to the Navier-Stokes equations
(submitted), arXiv:2508.01009.

    \bibitem{BKO}
   Bradshaw, Z., Kukavica, I. and O\.za\'nski, W.~S., 2022. Global weak solutions of the Navier-Stokes equations for intermittent initial data in half-space. Arch. Ration. Mech. Anal.~245, pp.~321--371.
   
   
   \bibitem{B}
   Bogovski\u{\i}, M.E., 1980. Solutions of some problems of vector analysis, associated with the operators div and grad. Theory of cubature formulas and the application of functional analysis to problems of mathematical physics, 1980, pp.~5--40.
   
   \bibitem{CKN}
   Caffarelli, L., Kohn, R. and Nirenberg, L., 1982. Partial regularity of suitable weak solutions of the Navier‐Stokes equations. Communications on pure and applied mathematics, 35(6), pp.~771--831.
   
   \bibitem{GS}
   Gallay, T. and Slijepčević, S., 2015. Uniform boundedness and long-time asymptotics for the two-dimensional Navier–Stokes equations in an infinite cylinder. Journal of mathematical fluid mechanics, 17, pp.~23--46.
   
   \bibitem{GSv}
   Guillod, J. and Šverák, V., 2023. Numerical investigations of non-uniqueness for the Navier–Stokes initial value problem in borderline spaces. Journal of Mathematical Fluid Mechanics, 25(3), p.46.
   
   \bibitem{JS}
  Jia, H. and Šverák, V., 2014. Local-in-space estimates near initial time for weak solutions of the Navier-Stokes equations and forward self-similar solutions. Inventiones mathematicae, 196, pp.~233--265.
   
   \bibitem{JS2}
   Jia, H. and Sverak, V., 2015. Are the incompressible 3d Navier–Stokes equations locally ill-posed in the natural energy space?. Journal of Functional Analysis, 268(12), pp.~3734--3766.
         
   \bibitem{kikuchi}
   Kikuchi, N. and Seregin, G., 2007. Weak solutions to the Cauchy problem for the Navier-Stokes equations satisfying the local energy inequality. Nonlinear equations and spectral theory, 220, p.141.
   
   \bibitem{KT}
   Kang, K., Miura, H. and Tsai, T.P., 2021. Short time regularity of Navier–Stokes flows with locally L3 initial data and applications. International Mathematics Research Notices, 2021(11), pp.~8763--8805.
   
\bibitem{K}
   Kukavica, I., 2008. On partial regularity for the Navier–Stokes equations. Discrete Contin. Dyn. Syst, 21(3), pp.~717--728.

\bibitem{K2}
Kukavica, I., 2003.
On local uniqueness of solutions of the Navier-Stokes
equations with bounded initial data.
J.~Diff.~Eq.~194, pp.~39--50.

\bibitem{KV}
Kukavica, I.~and Vicol, V., 2008.
On local uniqueness of weak solutions to
the Navier-Stokes system with $BMO^{-1}$ initial datum.
J.~Dynam.~Differential Equations~20, pp.~719--732.

   \bibitem{KwT} 
   Kwon, H. and Tsai, T.P., 2020. Global Navier–Stokes flows for non-decaying initial data with slowly decaying oscillation. Communications in Mathematical Physics, 375(3), pp.~1665--1715.
   
   \bibitem{16}
   Jia, H. and Sverak, V., 2013. Minimal $L^3$-initial data for potential Navier-Stokes singularities. SIAM Journal on Mathematical Analysis, 45(3), pp.~1448--1459.
   
   \bibitem{L}
   Leray, J., 1934. Sur le mouvement d'un liquide visqueux emplissant l'espace. Acta mathematica, 63, pp.~193--248.
   
   \bibitem{LR}
   Lemari\'e-Rieusset, P.G., 2002. Recent developments in the Navier-Stokes problem. Chapman \& Hall/CRC.
   
   \bibitem{MMP}
   Maekawa, Y., Miura, H. and Prange, C., 2019. Local energy weak solutions for the Navier–Stokes equations in the half-space. Communications in Mathematical Physics, 367, pp.~517--580.
   
   \bibitem{MS}
   Maremonti, P. and Shimizu, S., 2018. Global existence of solutions to 2-D Navier–Stokes flow with non-decaying initial data in half-plane. Journal of Differential Equations, 265(10), pp.~5352--5383.
   
  \bibitem{OP}
   O\.za\'nski, W. S. and Pooley, B., 2018.  Leray's fundamental work on the Navier-Stokes equations: a modern review of ``Sur le mouvement d’un liquide visqueux emplissant l’espace'', Partial Differential Equations in Fluid Mechanics, London Math. Soc. Lecture Note Ser., 452, Cambridge Univ. Press, pp.~113--203.
      
   \bibitem{T}
   Tsai, T.P., 2018. Lectures on Navier-Stokes equations (Vol.~192). American Mathematical Soc.

   \end{thebibliography}
\end{document}